\documentclass[11pt,a4paper,twoside]{amsart}
\usepackage{algorithm,algpseudocode,bm,caption,color,comment,epsfig,graphicx,mathabx,multirow,slashbox,srcltx,subcaption,txfonts}
\usepackage{amsmath,amsfonts,amssymb,amsthm}
\usepackage[colorlinks=true]{hyperref}
\usepackage{geometry}
\newtheorem{theorem}{Theorem}[section]

\newtheorem{example}[theorem]{Example}
\newtheorem{lemma}[theorem]{Lemma}

\newtheorem{proposition}[theorem]{Proposition}

\theoremstyle{definition}
\newtheorem{remark}[theorem]{Remark}
\newtheorem*{acknowledgements}{Acknowledgements}

\numberwithin{equation}{section}

\newcommand{\Ah}{A_{h}}
\newcommand{\uh}{u_{h}}
\newcommand{\vh}{v_{h}}
\newcommand{\fh}{f_{h}}

\newcommand{\bL}{\bm{A}}

\newcommand{\bu}{\bm{u}}
\newcommand{\buh}{\bu_{h}}
\newcommand{\bv}{\bm{v}}

\newcommand{\bmf}{\bm{f}}

\newcommand{\bw}{\bm{w}}

\newcommand{\curl}{\mathrm{curl}} 
\newcommand{\bcurl}{\bm{\curl}} 
\newcommand{\dvg}{\mathrm{div}} 
\newcommand{\bgrad}{\bm{\mathrm{grad}}} 

\newcommand{\Hocurl}{H(\Omega,\curl)} 
\newcommand{\Hcurl}{H(\curl)} 
\newcommand{\Hodiv}{H(\Omega,\dvg)} 
\newcommand{\Hdiv}{H(\dvg)} 
\newcommand{\mX}{\mathcal{X}} 
\newcommand{\Hox}{H(\Omega,\mX)} 
\renewcommand{\l}{\ell}

\newcommand{\hA}{\hat{A}}
\newcommand{\mC}{\mathcal{C}}
\newcommand{\mD}{\mathcal{D}}
\newcommand{\mH}{\mathcal{A}} 
\newcommand{\mT}{\mathcal{T}}
\newcommand{\mTh}{\mathcal{T}_{h}} 
\newcommand{\mV}{\mathcal{V}} 
\newcommand{\mVh}{\mV_{h}} 
\newcommand{\N}[1][0]{{\mathrm{N}}^{#1}} 
\newcommand{\RTN}[1][0]{{\mathrm{RTN}}^{#1}}
\newcommand{\R}{\mathbb{R}}

\title{Robust Algebraic multilevel preconditioning in $\Hcurl$ and $\Hdiv$}

\author{S.~K. Tomar}
\address{Johann Radon Institute for Computational and Applied Mathematics,
Austrian Academy of Sciences,\\ Altenbergerstrasse 69, 4040 Linz, Austria}
\email{satyendra.tomar@ricam.oeaw.ac.at}
\date{\today}


\begin{document}
\maketitle

\begin{abstract}
An algebraic multilevel iteration method for solving system of linear algebraic equations arising in $\Hcurl$ and $\Hdiv$ spaces are presented. The algorithm is developed for the discrete problem obtained by using the space of lowest order Nedelec and Raviart-Thomas-Nedelec elements. The theoretical analysis of the method is based only on some algebraic sequences and generalized eigenvalues of local (element-wise) problems. In the hierarchical basis framework, explicit recursion formulae are derived to compute the element matrices and the constant $\gamma$ (which measures the quality of the space splitting) at any given level.
It is proved that the proposed method is robust with respect to the problem parameters, and is of optimal order complexity. Supporting numerical results, including the case when the parameters have jumps, are also presented.
\end{abstract}

\pagestyle{myheadings}
\thispagestyle{plain}

\section{Introduction}
\label{sec:Intro}
Consider the finite element discretization of variational problems related to the bilinear form
\begin{equation}
\mH(\bu, \bv) := \alpha (\bu, \bv) + \beta (\mX \bu, \mX \bv), \quad \alpha,\beta \in \R^{+},
\label{eq:ModProb}
\end{equation}
defined on the Hilbert space
\begin{equation}
\Hox := \{ \bv \in (L^{2}(\Omega))^{d} : \mX \bv \in L^{2}(\Omega) \}.
\label{eq:Def_Hcurl}
\end{equation}
Here $\Omega \subset \R^{d}$, $d = 2, 3$, is a Lipschitz domain, and $\mX$ is the $\curl$ operator for $d = 2$ and $\dvg$ operator for $d = 3$. Note that $\dvg ~\bv = \partial_{x} v_{1} + \partial_{y} v_{2} + \partial_{z} v_{3} $ is the divergence of a three-dimensional vector $\bv = [v_{1}, v_{2}, v_{3}]^{T}$, $\curl ~\bv = \partial_{x} v_{2} - \partial_{y} v_{1}$ is the scalar curl of a two-dimensional vector $\bv = [v_{1}, v_{2}]^{T}$, and $(\cdot , \cdot)$ denotes the inner-product in $L^{2}(\Omega)$. For $\alpha = \beta = 1$,  the bilinear form \eqref{eq:ModProb} is precisely the inner-product in $\Hox$.
The adjoint of operator $\mX$ is defined by
\begin{equation*}
\mX^{a} = 
\begin{cases}
\bcurl & \mathrm{for}~ \mX = \curl, d = 2 \\
-\bgrad & \mathrm{for}~ \mX = \dvg, d = 3
\end{cases},
\end{equation*}
where, for a scalar function $w$, $\bgrad ~w = [\partial_{x} w, \partial_{y} w, \partial_{z} w]^{T}$ (for three-dimensional problem), and $\bcurl ~w = [\partial_{y} w, -\partial_{x} w]^{T}$ (for two-dimensional problem).
Associated with the inner-product $\mH$, there exists a linear operator $\bL := \alpha \bm{I} + \beta \mX^{a} \mX$, which maps $\Hox$ onto its dual space, and is determined by the relation
\begin{equation}
(\bL \bu, \bv) = \mH(\bu, \bv), \quad \forall \bv \in \Hox,
\label{eq:Def_L}
\end{equation}
Given a finite element space $\mVh$ of $\Hox$, the symmetric and positive-definite (SPD) operator $\Ah: \mVh \rightarrow \mVh$, which is the discretization of the operator $\bL$ together with natural boundary conditions, is determined by
\begin{equation}
(\Ah \uh , \vh) = \mH(\uh , \vh), \quad \forall \vh \in \mVh.
\label{eq:Def_Lh}
\end{equation}
The operator equation $\bL \bu = \bmf$, for $\bmf \in (L^{2}(\Omega))^{d}$, then leads to the following discrete problem
\begin{equation}
\Ah \uh = \fh ,
\label{eq:ModProbDiscrete}
\end{equation}
which is uniquely solvable.
For $\Hocurl$, such problems frequently occur in various contexts in electromagnetism, e.g., low-frequency time-harmonic Maxwell equations \cite{Monk-03}, or some formulations of the (Navier -) Stokes equations \cite{GiraultRaviart-86}, and for $\Hodiv$ such problems frequently occur in, e.g., mixed formulations of elliptic problems, least-squares formulations of elliptic problems, part of fluid flow problems, and in functional-type a posteriori error estimates, see, e.g., \cite{ArnoldFW-97, ArnoldFW-00, LazarovRT-08, Repin-08, RepinTomar-11} and the reference therein. Therefore, developing fast solvers for large system of equations \eqref{eq:ModProbDiscrete} is of significant importance.
Preconditioning methods for such linear systems in $\Hcurl$ within the framework of domain decomposition methods, multigrid methods, and auxiliary space methods have been proposed by several authors, see e.g., \cite{ArnoldFW-97, ArnoldFW-00, Hiptmair-97, HiptmairXu-07, VassilevskiLazarov-96, VassilevskiWang-92} and the references therein. The first results for multigrid in $\Hdiv$ (based on smoothing and approximation property) was presented in \cite{Brenner-92} for triangular elements. The first results for multigrid in $\Hcurl$ (within the framework of overlapping Schwarz methods) were obtained by Hiptmair in \cite{Hiptmair-98}. A unified treatment of multigrid methods for $\Hcurl$ and $\Hdiv$ was presented by Hiptmair and Toselli in \cite{HiptmairT-98}. However, the condition number estimates of their preconditioned system were not robust with respect to the parameters $\alpha$ and $\beta$. Arnold et al. \cite{ArnoldFW-00} employed the multigrid framework by developing necessary estimates for mixed finite element methods (FEM) based on discretizations of $\Hdiv$ and $\Hcurl$, and thereby obtained parameter independent condition number estimates of the preconditioned system.
Pasciak and Zhao studied the overlapping Schwarz methods for $\Hcurl$ in polyhedral domains in \cite{PasciakZ-02}, and Reitzinger and Schoeberl studied algebraic multigrid methods for edge elements in \cite{ReitzingerS-02}.
Auxiliary space preconditioning, proposed by Xu in \cite{Xu-96}, was studied for $H_{0}(\Omega,\curl)$ (the space $\Hcurl$ with zero tangential trace) by Hiptmair et al. \cite{HiptmairWZ-06}. Nodal auxiliary space preconditioning in $\Hcurl$ and $\Hdiv$ was studied by Hiptmair and Xu in \cite{HiptmairXu-07}, and the proposed preconditioner was robust with respect to the parameters $\alpha$ and $\beta$.
The main principles in constructing efficient multigrid and multilevel solvers for \eqref{eq:ModProbDiscrete} are projections into spaces of divergence-free vector fields, see \cite{VassilevskiWang-92}, or, alternatively, a discrete version of the Helmholtz decomposition, see e.g., \cite{ArnoldFW-97}, and/or the construction of a proper auxiliary space, see e.g., \cite{HiptmairXu-07}. Moreover, an effective error reduction generally demands to complement the coarse-grid correction by an appropriate smoother, e.g., additive or multiplicative Schwarz smoother, cf. \cite{ArnoldFW-00}. The simple scalar (point-wise) smoothers, in general, do not work satisfactorily for this class of problems.
All of these methods may be viewed as subspace correction methods \cite{Xu-92, XuZikatanov-02}, where different choices of specific components result in different methods (which also applies to the method presented in this paper).
Algebraic multilevel iteration (AMLI) methods were introduced by Axelsson and Vassilevski in a series of papers \cite{AxelssonV-89, AxelssonV-90, AxelssonV-91, AxelssonV-94}. The AMLI methods, which are recursive extensions of two-level methods for FEM \cite{AxelssonG-83}, have been extensively analyzed in the context of conforming and nonconforming FEM (including discontinuous Galerkin methods), see \cite{BlahetaMN-04, BlahetaMN-05, GKrausM-08, Kraus-02, KrausTomar-08, KrausTomar-08-3D, KrausTomar-11, Notay-00, Notay-02, NotayV-08}. For a detailed systematic exposition of AMLI methods, see the monographs \cite{KrausMargenov-09, Vassilevski-08}. These methods utilize a sequence of coarse-grid problems that are obtained from repeated application of a natural (and simple) hierarchical basis transformation, which is computationally advantageous. The underlying technique of these methods often requires only a few minor adjustments (mainly two-level hierarchical basis transformation) even if the underlying problem changes significantly. This is evident from the two different kind of problems considered in this paper, where the same algorithms (see Section~\ref{sec:Alg}) are used. Furthermore, the AMLI methods are robust with respect to the jumps in the operator coefficients (where classical multigrid methods suffer), and are computationally advantageous than classical algebraic multigrid methods.
In this paper, we first derive the results for two-dimensional $\Hcurl$ problem. Note that, in two-dimensions, the lowest-order Nedelec space can be obtained by a $90$ degrees rotation of lowest-order Raviart-Thomas space. Therefore, the space splitting presented in \cite{KrausTomar-11} also applies in this case (and vice-versa). However, we present a unified treatment of the element matrices arising from $(\bu, \bv)$ and $(\mX \bu, \mX \bv)$, which helps in deriving the explicit recursion formulae in simpler forms and without any undetermined constants. Moreover, with the unified treatment we are able to extend the results to three-dimensional lowest-order Raviart-Thomas-Nedelec elements in a straight-forward manner. Our analysis is based only on some algebraic sequences and the generalized eigenvalues of local (element-wise) problems. In hierarchical setting, we derive explicit recursion formulae to compute the element matrices and the constant $\gamma$ (which measures the quality of the space splitting) at any given level. The method is shown to be robust with respect to the parameters, i.e., the results hold uniformly for $0 < \alpha , \beta < \infty$.
The remainder of this paper is organized as follows. In Section~\ref{sec:DiscreteProb} we briefly discuss the finite element discretization of the model problem (\ref{eq:ModProb}) using the lowest-order Nedelec and Raviart-Thomas-Nedelec spaces. Section~\ref{sec:AMLI} starts with a brief description of the AMLI procedure (in Section~\ref{sec:Prec}). After presenting hierarchical basis transformations in Section~\ref{sec:HB}, the construction of the hierarchical splitting of the lowest-order Nedelec and Raviart-Thomas-Nedelec spaces is presented in Section~\ref{sec:Splitting}. In Section~\ref{sec:LocalAnalysis} a local two- and multi-level analysis is then presented and the main result is proved. The algorithms used in this paper are provided in Section~\ref{sec:Alg}. Finally, in Section~\ref{sec:NumRes} we present numerical experiments. These include the cases with known analytical solution ($\alpha = \beta = 1$), fixing one of the parameters and varying other from $10^{-6}$ to $10^{6}$, and the case of jumping coefficients. The conclusions are drawn in Section~\ref{sec:Conclusion}.

\section{Finite element discretization}
\label{sec:DiscreteProb}

In this section we briefly discuss the finite element discretization using lowest order Nedelec space in two-dimensions and lowest-order Raviart-Thomas-Nedelec space in three-dimensions, respectively.

\subsection{Finite element discretization using Nedelec elements}
\label{sec:DiscreteN}

We consider the tessellation of $\Omega \subset \R^{2}$ using square elements, and choose the reference element $\hat{K}$ as $[-1,1] \times [-1,1]$. Let $P_{r_{x},r_{y}}(\hat{K})$ denote the space of polynomials of degree $\leq r_{x}$ in $x$ and $\leq r_{y}$ in $y$. Also, let $P_{r}(\partial \hat{K})$ denote the space of polynomials of degree $\leq r$ on $\partial \hat{K}$. For the construction of $\mVh$, we use the space of lowest-order edge elements (Nedelec space of first kind), which is denoted by $\N$. The space $\N(\hat{K})$ is defined as
\begin{align}
\label{eq:N0}
\N(\hat{K})
& = P_{0,1}(\hat{K}) \times P_{1,0}(\hat{K})
= \left\{
\bm{v}(\hat{x},\hat{y})
= \left[ \begin{array}{c}
v_{1} + v_{2} \hat{y} \\
v_{3} + v_{4} \hat{x}
\end{array} \right]
\right\}.
\end{align}
Thus, the local basis for $\N$ has dimension $4$. Moreover, for $\bm{v}_{0} \in \N(\hat{K})$ we have
\begin{equation}
\curl ~\bm{v}_{0} \in P_{0,0}~, \quad \bm{v}_{0} \cdot \bm{t} \vert_{\partial \hat{K}} \in P_{0}(\partial \hat{K}),
\end{equation}
where $\bm{t}$ denotes the unit tangential vector to the element boundaries. For further details the reader is referred to, e.g., \cite{Monk-03}.
Now let $F: \hat{K} \rightarrow \mathbb{R}^{2}$ be a diffeomorphism of the reference element $\hat{K}$ onto a physical element $K$, i.e.,
$K = F(\hat{K})$. By $\mathcal{J}$ we denote the Jacobian matrix of the mapping, and by $\mathcal{J}_{D}$ its determinant, which are defined as
\begin{equation*}
\mathcal{J} =
\left( \begin{array}{cc}
\partial_{\hat{x}} x & \partial_{\hat{y}} x\\
\partial_{\hat{x}} y & \partial_{\hat{y}} y
\end{array}
\right),
\quad
\mathcal{J}_{D} = \vert \mathrm{det} \mathcal{J}\vert = \partial_{\hat{x}} x ~\partial_{\hat{y}} y - \partial_{\hat{y}} x ~\partial_{\hat{x}} y > 0.
\end{equation*}
Then we have the following transformation relations:
\begin{align}
\bw = \mathcal{J}^{-T} \bm{\hat{w}}; \quad
\curl ~\bw = \mathcal{J}_{D}^{-1} \curl ~\bm{\hat{w}},
\quad \forall \bm{w} \in H(K, \curl \!), \bm{\hat{w}} \in H(\hat{K}, \curl \!).
\label{eq:VecTrans_2D}
\end{align}
The vector transformation $\bm{w} \rightarrow \mathcal{J}^{-T} \bm{\hat{w}}$ is called the covariant transformation, and $\curl ~\bw = \mathcal{J}_{D}^{-1} \curl ~\bm{\hat{w}}$ is obtained via the well known Piola transformation $\bm{w} \rightarrow \mathcal{J}_{D}^{-1} \mathcal{J} \bm{\hat{w}}$.
We denote the element matrix for $\int_{K} \bu \cdot \bv $ by $L_{K}$, and for $\int_{K} \curl ~\bu ~ \curl ~\bv $ by $C_{K}$. For the $\N$ space based on uniform mesh composed of square elements, the element matrices $L_{K}$ and $C_{K}$ have the following structure
\begin{align}
L_{K} = \frac{1}{6} \left [
\begin{array}{rrrr}
2 & 1 & 0 & 0 \\
1 & 2 & 0 & 0 \\
0 & 0 & 2 & 1 \\
0 & 0 & 1 & 2 \\
\end{array}
\right ],
\quad
C_{K} = \frac{1}{h^{2}} \left [
\begin{array}{rrrr}
1 & -1 & -1 & 1 \\
-1 & 1 & 1 & -1 \\
-1 & 1 & 1 & -1 \\
1 & -1 & -1 & 1 \\
\end{array}
\right ].
\label{eq:ElmMatLkCk}
\end{align}
The overall element matrix $A_{K, C} := \alpha L_{K} + \beta C_{K}$, is thus given by
\begin{align}
A_{K, C} = \frac{1}{6 h^{2}} \left [
\begin{array}{cccc}
2 \alpha h^{2} + 6 \beta & \alpha h^{2} -6 \beta & -6 \beta & 6 \beta \\
\alpha h^{2} -6 \beta & 2 \alpha h^{2} + 6 \beta & 6 \beta & -6 \beta \\
-6 \beta & 6 \beta & 2 \alpha h^{2} + 6 \beta & \alpha h^{2} -6 \beta \\
6 \beta & -6 \beta & \alpha h^{2} -6 \beta & 2 \alpha h^{2} + 6 \beta \\
\end{array}
\right ].
\label{eq:ElmMatAk_2D}
\end{align}
Letting $e = \kappa h^{2}$, with $\kappa = \alpha/\beta$, the element matrix can be written as
\begin{align}
A_{K, C} = \frac{\beta}{6 h^{2}} \left [
\begin{array}{cccc}
2e+6 & e-6 & -6 & 6 \\
e-6 & 2e+6 & 6 & -6 \\
-6 & 6 & 2e+6 & e-6 \\
6 & -6 & e-6 & 2e+6 \\
\end{array}
\right ].
\label{eq:ElmMatAk_k_2D}
\end{align}
Clearly, for all $\alpha, \beta \in \R^{+}$, and thus $\kappa \in \R^{+}$, we have $e > 0$. Note that for fixed $\kappa$, and $h \rightarrow 0$, the element matrix $A_{K, C}$ is dominated by the matrix $C_{K}$ (which has a non-zero kernel), whereas for moderate values of $h$ it is a regular matrix. The near-nullspace of the matrix $A_{K, C}$ is given by the nullspace of the matrix $C_{K}$, which is associated with the local bilinear form $\mC_K (\bm{u},\bm{v}) := (\curl ~\bm{u}, \curl ~\bm{v})_K.$ As we shall see in the analysis, the proposed method is of optimal order for all $0 < \alpha, \beta < \infty$.
The following result can now be easily shown using \cite[Lemma~2.1]{KrausTomar-11}.
\begin{lemma}{\rm (Near-nullspace of matrix $A_{K, C}$).}\label{lem1_2D}
The element matrix $A_{K, C}$ given in \eqref{eq:ElmMatAk_2D} is symmetric positive definite (SPD). Moreover, the nullspace of the matrix $C_K$ for a general element $K$ with nodal coordinates $(x_i,y_i)$, $i \in \{1,2,3,4\}$ is given by
\begin{equation}
\ker (C_K) = {\rm span} \{ (1,1,0,0)^T,(0,0,1,1)^T,(x_1, x_2, y_3, y_4)^T \} .
\end{equation}
Furthermore, in case of a uniform mesh composed of square $\N$ elements, the matrix $C_K$ is same for each element $K$ and its nullspace is given by
\[\ker (C_K) = {\rm span} \{ (1,1,0,0)^T, (0,0,1,1)^T, (-1,0,0,1)^T \}.\]
\end{lemma}
\begin{remark}\label{remC_2D}
When using the lowest order Nedelec elements, the matrix $C_K$ is always of rank one. In the global assembly this yields a matrix $C$ whose rank equals the number of elements in the mesh. That is, the kernel of the global matrix $C$ has dimension ${\rm dim} ( {\rm ker} (C)) = n_E - n_K$, where $n_E$ denotes the number of faces and $n_K$ denotes the number of elements in the finite element mesh.
Thereby, the dimension of the kernel is slightly more than half of the total number of degrees of freedom.
\end{remark}
\subsection{Finite element discretization using Raviart-Thomas-Nedelec elements}
\label{sec:DiscreteRTN}

We consider the tessellation of $\Omega \subset \R^{3}$ using cubic elements, and choose the reference element $\hat{K}$ as $[-1,1]^{3}$. Let $P_{r_{x},r_{y},r_{z}}(\hat{K})$ denote the space of polynomials of degree $\leq r_{x}$ in $x$, $\leq r_{y}$ in $y$ and $\leq r_{z}$ in $z$, respectively. Also, let $P_{r_{1},r_{2}}(\partial \hat{K})$ denote the space of polynomials of degrees $\leq r_{1}$ and $\leq r_{2}$ in the respective dimensions on $\partial \hat{K}$. For the construction of $\mVh$, we use the space of lowest-order Raviart-Thomas-Nedelec elements, which is denoted by $\RTN$. The space $\RTN(\hat{K})$ is defined as
\begin{align}
\label{eq:RTN0}
\RTN(\hat{K})
& = P_{1,0,0}(\hat{K}) \times P_{0,1,0}(\hat{K}) \times P_{0,0,1}(\hat{K})
= \left\{
\bm{v}(\hat{x},\hat{y},\hat{z})
= \left[ \begin{array}{c}
v_{1} + v_{2} \hat{x} \\
v_{3} + v_{4} \hat{y} \\
v_{5} + v_{6} \hat{z}
\end{array} \right]
\right\}.
\end{align}
Thus, the local basis for $\RTN$ has dimension $6$. Moreover, for $\bm{v}_{0} \in \RTN(\hat{K})$ we have
\begin{equation}
\dvg ~\bm{v}_{0} \in P_{0,0,0}~, \quad \bm{v}_{0} \cdot \bm{n} \vert_{\partial \hat{K}} \in P_{0,0}(\partial \hat{K}),
\end{equation}
where $\bm{n}$ denotes the unit normal vector to the element faces. For further details the reader is referred to, e.g., \cite{BrezziFortin-91}.
Now let $F: \hat{K} \rightarrow \mathbb{R}^{3}$ be a diffeomorphism of the reference element $\hat{K}$ onto a physical element $K$, i.e.,
$K = F(\hat{K})$. By $\mathcal{J}$ we denote the Jacobian matrix of the mapping, and by $\mathcal{J}_{D}$ its determinant, which are defined as
\begin{equation*}
\mathcal{J} =
\left( \begin{array}{ccc}
\partial_{\hat{x}} x & \partial_{\hat{y}} x & \partial_{\hat{z}} x \\
\partial_{\hat{x}} y & \partial_{\hat{y}} y & \partial_{\hat{z}} y \\
\partial_{\hat{x}} z & \partial_{\hat{y}} z & \partial_{\hat{z}} z
\end{array}
\right),
\quad
\mathcal{J}_{D} = \vert \mathrm{det} \mathcal{J}\vert > 0.
\end{equation*}
Then we have the following relations:
\begin{align}
\bw = \mathcal{J}_{D}^{-1} \mathcal{J} \bm{\hat{w}}; \quad
\dvg ~\bw = \mathcal{J}_{D}^{-1} \dvg ~\bm{\hat{w}},
\quad \forall \bm{w} \in H(K, \dvg), \bm{\hat{w}} \in H(\hat{K}, \dvg),
\label{eq:VecTrans_3D}
\end{align}
by the well known Piola transformation, see e.g., \cite{BrezziFortin-91}.
We denote the element matrix for $\int_{K} \bu \cdot \bv $ by $L_{K}$, and for $\int_{K} \dvg ~\bu ~ \dvg ~\bv $ by $D_{K}$. For the $\RTN$ space based on uniform mesh composed of cubic elements, the element matrices $L_{K}$ and $D_{K}$ have the following structure
\begin{align}
L_{K} = \frac{1}{6h} \left [
\begin{array}{rrrrrr}
2 & 1 & 0 & 0 & 0 & 0 \\
1 & 2 & 0 & 0 & 0 & 0 \\
0 & 0 & 2 & 1 & 0 & 0 \\
0 & 0 & 1 & 2 & 0 & 0 \\
0 & 0 & 0 & 0 & 2 & 1 \\
0 & 0 & 0 & 0 & 1 & 2
\end{array}
\right ],
\quad
D_{K} = \frac{1}{h^{3}} \left [
\begin{array}{rrrrrr}
1 & -1 & 1 & -1 & 1 & -1 \\
-1 & 1 & -1 & 1 & -1 & 1 \\
1 & -1 & 1 & -1 & 1 & -1 \\
-1 & 1 & -1 & 1 & -1 & 1 \\
1 & -1 & 1 & -1 & 1 & -1 \\
-1 & 1 & -1 & 1 & -1 & 1
\end{array}
\right ].
\label{eq:ElmMatBkCk}
\end{align}
The overall element matrix $A_{K, D} := \alpha L_{K} + \beta D_{K}$, is thus given by
\begin{align}
A_{K, D} = \frac{1}{6 h^{3}} \left [
\begin{array}{cccccc}
2 \alpha h^{2} + 6 \beta & \alpha h^{2} -6 \beta & 6 \beta & -6 \beta & 6 \beta & -6 \beta \\
\alpha h^{2} -6 \beta & 2 \alpha h^{2} + 6 \beta & -6 \beta & 6 \beta & -6 \beta & 6 \beta \\
6 \beta & -6 \beta & 2 \alpha h^{2} + 6 \beta & \alpha h^{2} -6 \beta & 6 \beta & -6 \beta \\
-6 \beta & 6 \beta & \alpha h^{2} -6 \beta & 2 \alpha h^{2} + 6 \beta & -6 \beta & 6 \beta \\
6 \beta & -6 \beta & 6 \beta & -6 \beta & 2 \alpha h^{2} + 6 \beta & \alpha h^{2} -6 \beta \\
-6 \beta & 6 \beta & -6 \beta & 6 \beta & \alpha h^{2} -6 \beta & 2 \alpha h^{2} + 6 \beta
\end{array}
\right ].
\label{eq:ElmMatAk_3D}
\end{align}
With the definition of $e$ introduced before \eqref{eq:ElmMatAk_k_2D}, the element matrix can be written as
\begin{align}
A_{K, D} = \frac{\beta}{6 h^{3}} \left [
\begin{array}{cccccc}
2e+6 & e-6 & 6 & -6 & 6 & -6 \\
e-6 & 2e+6 & -6 & 6 & -6 & 6 \\
6 & -6 & 2e+6 & e-6 & 6 & -6 \\
-6 & 6 & e-6 & 2e+6 & -6 & 6 \\
6 & -6 & 6 & -6 & 2e+6 & e-6 \\
-6 & 6 & -6 & 6 & e-6 & 2e+6
\end{array}
\right ].
\label{eq:ElmMatAk_k_3D}
\end{align}
Note again that for fixed $\kappa$, and $h \rightarrow 0$, the element matrix $A_{K, D}$ is dominated by the matrix $D_{K}$ (which has a non-zero kernel), whereas for moderate values of $h$ it is a regular matrix. The near-nullspace of the matrix $A_{K, D}$ is given by the nullspace of the matrix $D_{K}$, which is associated with the local bilinear form $\mD_K (\bm{u},\bm{v}) := (\dvg ~\bm{u}, \dvg ~\bm{v})_K.$ As we shall see in the analysis, the proposed method is of optimal order for all $0 < \alpha, \beta < \infty$.
\begin{proposition}{\rm (Near-nullspace of matrix $A_{K, D}$).}\label{lem1_3D}
The element matrix $A_{K, D}$ given in \eqref{eq:ElmMatAk_3D} is symmetric positive definite (SPD).
Furthermore, in case of a uniform mesh composed of cubic $\RTN$ elements, the matrix $D_K$ is same for each element $K$ and its nullspace is given by
\begin{align*}
& \ker (D_K) \\
= & {\rm span} \{ (1, 1, 0, 0, 0, 0)^{T}, (-1, 0, 1, 0, 0, 0)^{T}, (1, 0, 0, 1, 0, 0)^{T}, (-1, 0, 0, 0, 1, 0)^{T}, (1, 0, 0, 0, 0, 1)^{T} \}.
\end{align*}
\end{proposition}
\begin{proof}
Since the coefficients $\alpha$ and $\beta$ in \eqref{eq:ElmMatAk_3D} are positive, it follows from equation (\ref{eq:ModProbDiscrete}) that $A_{K, D}$ is SPD for a general element $K$.
Moreover, for a uniform mesh composed of cubic $\RTN$ elements, since the vector $(1,-1,1,-1,1,-1)^{T}$ is orthogonal to the kernel of $D_{K}$, it is clear that the rank-one matrix $D_{K}$ is of the form $c \cdot (1,-1,1,-1,1,-1)^T \cdot (1,-1,1,-1,1,-1)$, for some constant $c$.
\end{proof}

\begin{remark}\label{remC_3D}
When using the lowest order Raviart-Thomas-Nedelec elements, the matrix $D_{K}$ is always of rank one. In the global assembly this yields a matrix $D$ whose rank equals the number of elements in the mesh. That is, the kernel of the global matrix $D$ has dimension ${\rm dim} ( {\rm ker} (D)) = n_{F} - n_{K}$, where $n_{F}$ denotes the number of faces and $n_{K}$ denotes the number of elements in the finite element mesh.
Thereby, the dimension of the kernel is slightly more than two-third of the total number of degrees of freedom.
\end{remark}
\section{Algebraic multilevel iteration}
\label{sec:AMLI}

For the solution of the linear system arising from \eqref{eq:ModProbDiscrete}, we describe and analyze the AMLI method in the remainder of this section. Our presentation follows \cite{KrausTomar-11}.

\subsection{The AMLI procedure}
\label{sec:Prec}

In what follows we will denote by $M^{(\l)}$ a preconditioner for a finite element (stiffness) matrix $A^{(\l)}$ corresponding to a $\l$ times refined mesh $(0 \le \l \le L)$. We will also make use of the corresponding $\l^{\mathrm{th}}$ level hierarchical matrix $\hA^{(\l)}$, which is related to $A^{(\l)}$ via a two-level hierarchical basis (HB) transformation $J^{(\l)}$, i.e.,
\begin{equation}\label{NBHB}
\hA^{(\l)} = J^{(\l)} A^{(\l)} (J^{(\l)})^T .
\end{equation}
The transformation matrix $J^{(\l)}$ specifies the space splitting, and will be described in detail in Section~\ref{sec:HB}. By  $A^{(\l)}_{ij}$ and $\hA^{(\l)}_{ij}$, $1 \le i,j \le 2$, we denote the blocks of $A^{(\l)}$ and $\hA^{(\l)}$ that correspond to the fine-coarse partitioning of degrees of freedom (DOF) where the DOF associated with the coarse mesh are numbered last.
The aim is to build a multilevel preconditioner $M^{(L)}$ for the coefficient matrix $A^{(L)}:= A_h$ at the level of the finest mesh that has a uniformly bounded (relative) condition number
\[
\varkappa({M^{(L)}}^{-1} A^{(L)}) = {\mathcal O}(1),
\]
and an optimal computational complexity, that is, linear in the number of degrees of freedom $N_{L}$ at the finest mesh (grid). In order to achieve this goal hierarchical basis methods can be combined with various types of stabilization techniques.
One particular purely algebraic stabilization technique is the so-called Algebraic  Multi-Level Iteration (AMLI) method, which is presented below.
We have the following two-level hierarchical basis representation at level $\l$
\begin{equation}
\hat{A}^{(\l)} =
{\def\arraystretch{1.4}
\begin{bmatrix}
\hat{A}_{11}^{(\l)}&\hat{A}_{12}^{(\l)}\\
\hat{A}_{21}^{(\l)}&\hat{A}_{22}^{(\l)}\end{bmatrix}
= \begin{bmatrix}
\hat{A}_{11}^{(\l)}&\hat{A}_{12}^{(\l)}\\
\hat{A}_{21}^{(\l)}&A^{(\l-1)}\end{bmatrix}} .
\label{block_H}
\end{equation}
Starting at level $0$ (associated with the coarsest mesh), on which a complete LU factorization of the matrix $A^{(0)}$ is performed, we define
\begin{equation}\label{ml}
M^{(0)}:=A^{(0)}.
\end{equation}
Given the preconditioner $M^{(\l-1)}$ at level $\l-1$, the preconditioner $M^{(\l)}$ at level $\l$ is then defined by
\begin{equation}\label{mk}
M^{(\l)}:=L^{(\l)} U^{(\l)},
\end{equation}
where
\begin{equation}\label{rklk}
{\def\arraystretch{1.4}
L^{(\l)}:=\left[ \begin{array}{cc}
C_{11}^{(\l)} & 0 \\
\hA_{21}^{(\l)} & C_{22}^{(\l)}
\end{array} \right], \quad
U^{(\l)}:=\left[ \begin{array}{cc}
I & {C_{11}^{(\l)}}^{-1} \hA_{12}^{(\l)} \\
0 & I
\end{array} \right]} .
\end{equation}
Here $C_{11}^{(\l)}$ is a preconditioner for the pivot block $A_{11}^{(\l)}$, and
\begin{equation}\label{zk}
C_{22}^{(\l)} := {A}^{(\l-1)} \left( I - p^{(\l)} ( {M^{(\l-1)}}^{-1} A^{(\l-1)} ) \right)^{-1}
\end{equation}
is an approximation to the Schur complement $S=A^{(\l-1)}-\hA_{21}^{(\l)} {C_{11}^{(\l)}}^{-1} \hA_{12}^{(\l)}$, where $A^{(\l-1)} = \hA_{22}^{(\l)}$ is the stiffness matrix at the coarse level $\l-1$, and $p^{(\l)}$ is a certain stabilization polynomial of degree $\nu_\l$ satisfying the condition
\begin{equation}\label{pk}
0 \le p^{(\l)} (x) < 1, \quad \forall ~ 0 < x \le 1, \quad \mathrm{~and~} p^{(\l)} (0) = 1 .
\end{equation}
It is easily seen that (\ref{zk}) is equivalent to
\begin{equation}\label{zk1}
{C_{22}^{(\l)}}^{-1} =
{M^{(\l-1)}}^{-1} q^{(\l)} ( A^{(\l-1)} {M^{(\l-1)}}^{-1} ),
\end{equation}
where the polynomial $q^{(\l)}(x)$ is given by
\begin{equation}\label{q}
q^{(\l)}(x) = \frac{1 - p^{(\l)}(x)}{x}.
\end{equation}
We note that the multilevel preconditioner defined via (\ref{mk}) is getting close to a two-level method when $q^{(\l)}(x)$ closely approximates $1/x$, in which case ${C_{22}^{(\l)}}^{-1} \approx {A^{(\l-1)}}^{-1}$. In order to construct an efficient multilevel method the action of ${C_{22}^{(\l)}}^{-1}$ on an arbitrary vector should be much cheaper to compute (in terms of the number of arithmetic operations) than the action of ${A^{(\l-1)}}^{-1}$. Optimal order solution algorithms typically require that the arithmetic work for one application of ${C_{22}^{(\l)}}^{-1}$ is of the order ${\mathcal O}(N_{\l-1})$ where $N_{\l-1}$ denotes the number of unknowns at level $\l-1$. 
To reduce the overall complexity of AMLI methods (to achieve optimal computational complexity), various stabilization techniques can be used.
%
It is well known from the theory introduced in \cite{AxelssonV-89,AxelssonV-90} that a properly shifted and scaled Chebyshev polynomial $p^{(\l)} := p_{\nu_{\l}}$ of degree $\nu_{\l}$ can be used to stabilize the condition number of ${M^{(\ell)}}^{-1} A^{(\l)}$ (and thus obtain optimal order computational complexity). Other polynomials such as the best polynomial approximation of $1/x$ in uniform norm also qualify for stabilization, see, e.g., \cite{KrausVZ-12}.
This approach requires the computation of polynomial coefficients which depends on the bounds of the eigenvalues of the preconditioned system.
Alternatively, a few inner flexible conjugate gradient (FCG) type iterations are performed at coarse levels to stabilize (or freeze the residual reduction factor of) the outer FCG iteration, which lead to parameter-free AMLI methods \cite{AxelssonV-91, AxelssonV-94, Kraus-02, Notay-00, Notay-02, NotayV-08}. In general, the resulting nonlinear (variable step) multilevel preconditioning method is almost equally efficient as linear AMLI method, and, because its realization does not rely on any spectral bounds, it is easier to implement than the linear AMLI method (based on a stabilization polynomial). For a convergence analysis of nonlinear AMLI see, e.g., \cite{Kraus-02, KrausMargenov-09, Vassilevski-08}.
Typically, the iterative solution process is of optimal order of computational complexity if the degree $\nu_\l =\nu$ of the matrix polynomial (or alternatively, the number of inner iterations for nonlinear AMLI) at level $\l$ satisfies the optimality condition
\begin{alignat}{1}
1/\sqrt{(1- \gamma^{2})} < \, \nu  < \,\tau,
\label{eq:CondOpt}
\end{alignat}
where $\tau \approx \tau_\l = {N_{\l}}/{N_{\l-1}}$ denotes the reduction factor of the number of degrees of freedom (DOF), and $\gamma$ denotes the constant in the strengthened Cauchy-Bunyakowski-Schwarz (CBS) inequality. In case of standard (full) coarsening, the value of $\tau$ is approximately $4$ for the sequence of $\N$ spaces, and $8$ for the sequence of $\RTN$ spaces. These sequences will be constructed in the next subsections. For a more detailed discussion of AMLI methods, including implementation issues see, e.g., \cite{KrausMargenov-09, Vassilevski-08}.
\begin{remark}
The commonly used AMLI algorithm was originally introduced and studied in a multiplicative form \eqref{mk}, see \cite{AxelssonV-89, AxelssonV-90}. However, the preconditioner can also be constructed in the additive form, which is defined as follows \cite{Axelsson-99, AxelssonP-99, KrausMargenov-09}
\begin{equation}\label{mk_a}
M^{(\l)}_{A} := \left[ \begin{array}{cc}
C_{11}^{(\l)} & 0 \\
0 & C_{22}^{(\l)}
\end{array} \right].
\end{equation}
In this case the optimal order of computational complexity demands that the matrix polynomial degree (or the number of inner iterations of nonlinear AMLI) satisfy the following relation
\begin{alignat}{1}
\sqrt{(1 + \gamma)/(1- \gamma)} < \, \nu  < \,\tau.
\label{eq:CondOpt_Add}
\end{alignat}
\end{remark}
\subsection{Hierarchical basis for $\mVh$}
\label{sec:HB}

The AMLI methods we are considering here, for the solution of (\ref{eq:ModProbDiscrete}), are based on a proper splitting of the space $\mVh$. 
\begin{figure}[!ht]
\centering
\includegraphics[width=.5\textwidth, keepaspectratio, clip]{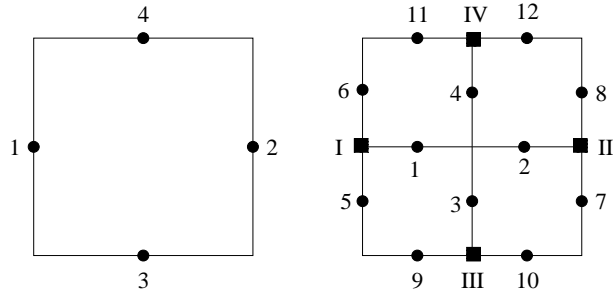}
\caption{Macro-element obtained after one regular mesh-refinement step}
\label{fig:MacroElm_Edges}
\end{figure}
For $\N$ subspace of $\Hcurl$, the particular two-level HB transformation that induces this splitting was introduced in the context of linear nonconforming Crouzeix-Raviart (CR) elements in \cite{BlahetaMN-04, BlahetaMN-05}. It was later studied for quadrilateral rotated bilinear (Rannacher-Turek) type elements in \cite{GKrausM-08}. Note that the similarities of the HB transformation when using CR elements and Nedelec elements is due to the algebraic nature of the problem. For the discretization based on linear elements (for meshes consisting of triangles) or bilinear elements (for meshes consisting of squares), similar HB transformation matrix can be used. However, suitable changes will be required when working with meshes consisting of general quadrilaterals.
Consider two consecutive discretizations $\mT_{H}$ (coarse level) and $\mTh$ (fine level). Figure~\ref{fig:MacroElm_Edges} illustrates a macro-element $G$ (at fine level) obtained from a coarse element by one regular mesh-refinement step.
Let  $\varphi_G=\{\phi_i(x,y)\}_{i=1}^{12}$ be the macro-element vector of the nodal basis functions. Using the local numbering of DOF, as shown in Figure~\ref{fig:MacroElm_Edges} (right picture), a macro-element level (local) transformation matrix $J_G$ is constructed based on differences and aggregates of each pair of basis functions $\phi_i$ and $\phi_j$ that correspond to a macro element edge, i.e.,
\begin{equation}\label{eq:J_G_N}
J_{G} = \frac12
\left [
\begin{array}{rrrrrrrrrrrr}
2&  &  &  &  &  &  &  &  &  &  &  \\
&  2&  &  &  &  &  &  &  &  &  &  \\
&  &  2&  &  &  &  &  &  &  &  &  \\
&  &   & 2&  &  &  &  &  &  &  &  \\
&  &  &  &  1& -1& &  &  &  &  &  \\
&  &  &  &  &  & 1& -1&  &  &  &  \\
&  &  &  &  &  &  &  & 1& -1&  &  \\
&  &  &  &  &  &  &  &  &  & 1&-1\\
&  &  &  & 1& 1&  &  &  &  &  &  \\
&  &  &  &  &  & 1& 1&  &  &  &  \\
&  &  &  &  &  &  &  & 1& 1&  &  \\
&  &  &  &  &  &  &  &  &  & 1& 1
\end{array}
\right ].
\end{equation}
%
%

%
\begin{figure}[!ht]
\centering
\includegraphics[width=.45\textwidth, keepaspectratio, clip]{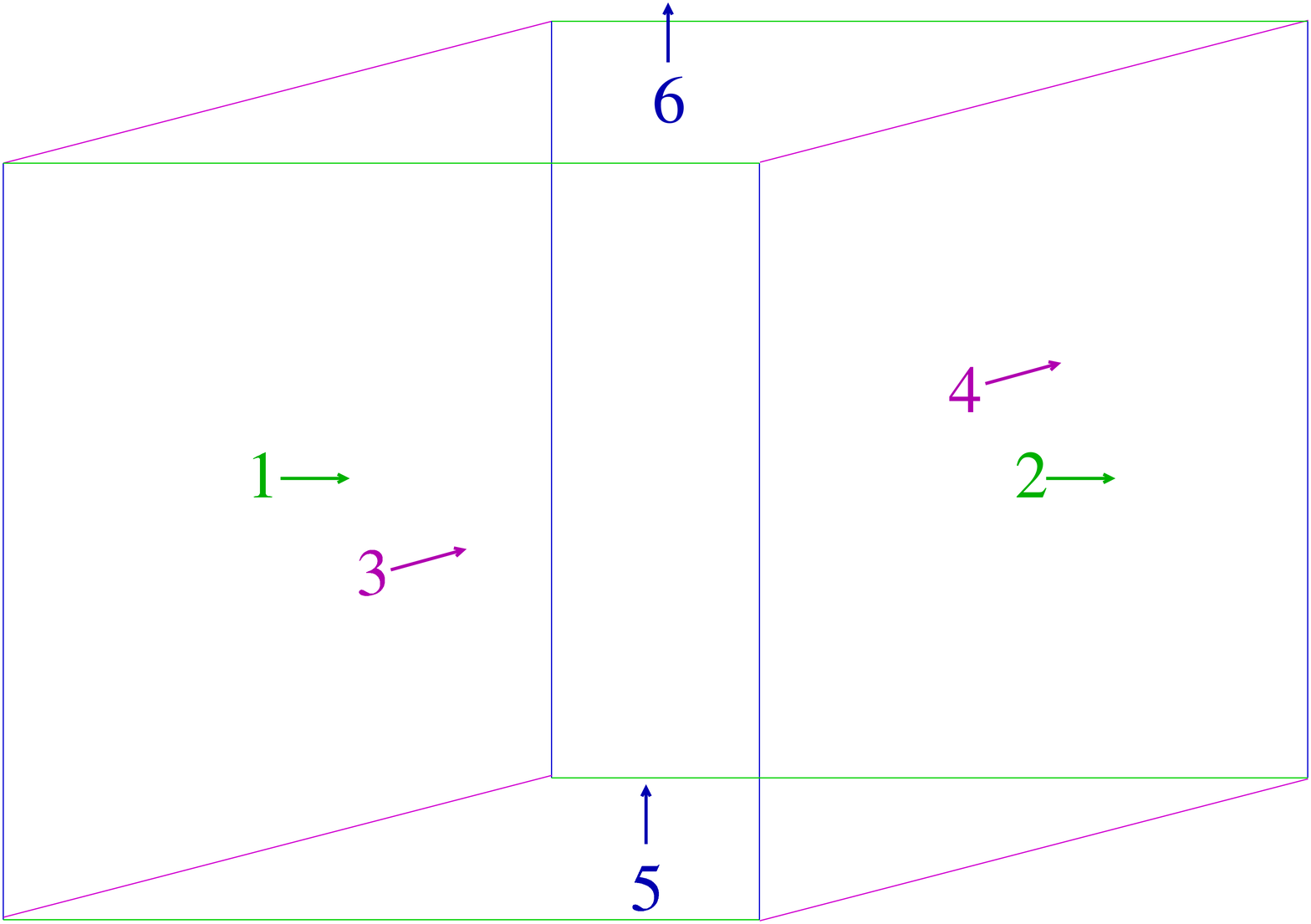} \hspace{10mm}
\includegraphics[width=.45\textwidth, keepaspectratio, clip]{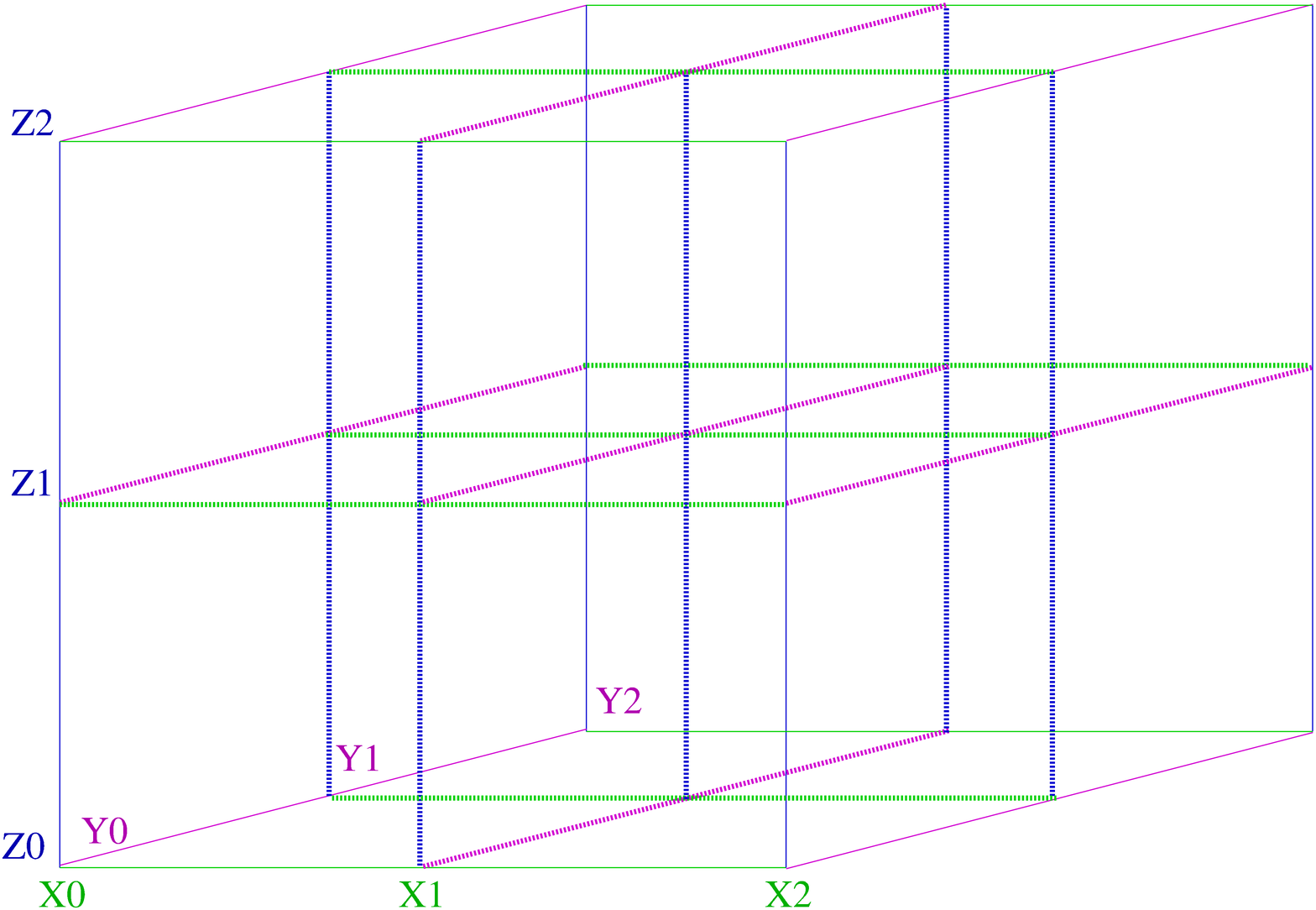}\\
\includegraphics[width=.2\textwidth, keepaspectratio, clip]{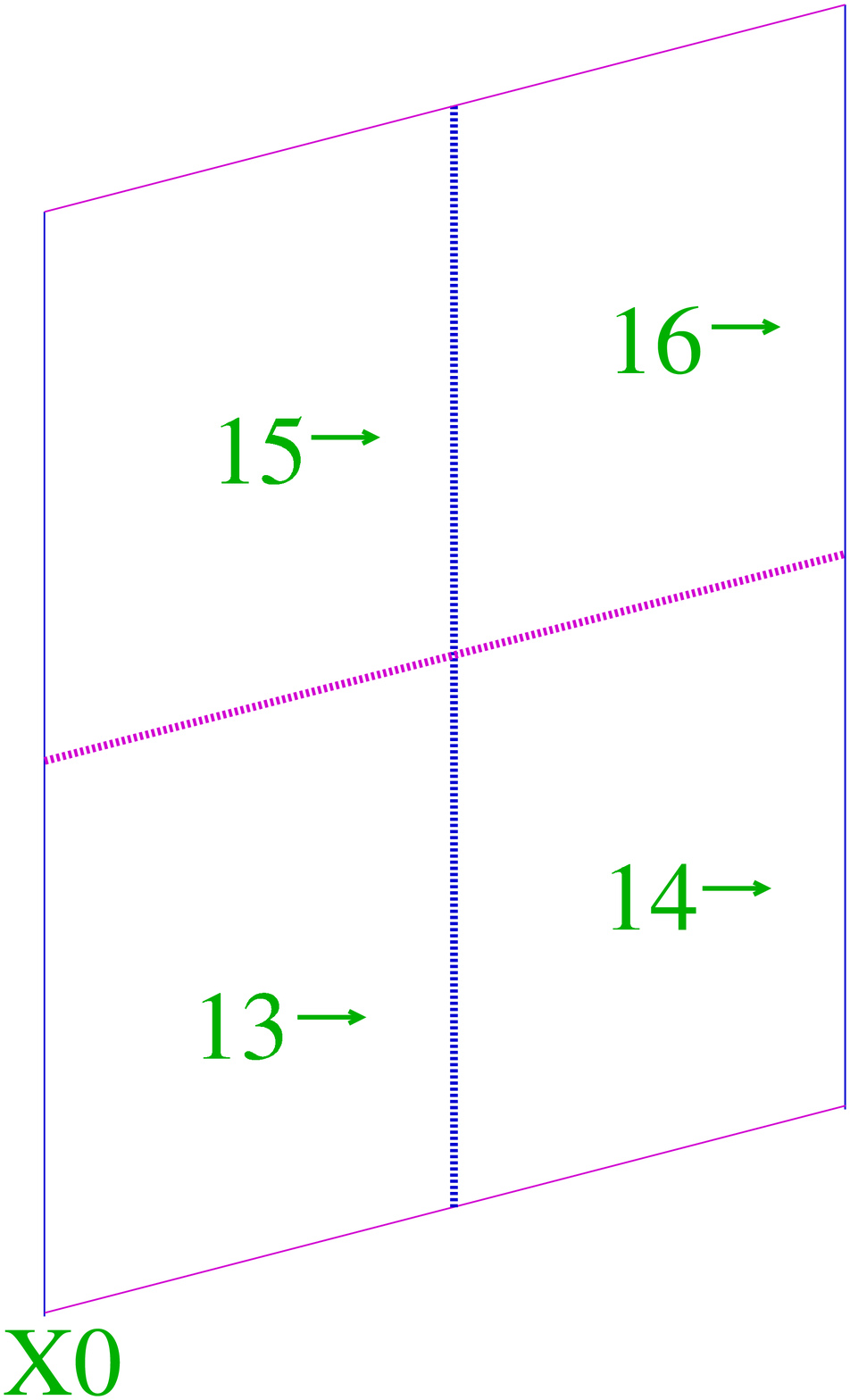} \hspace{5mm}
\includegraphics[width=.2\textwidth, keepaspectratio, clip]{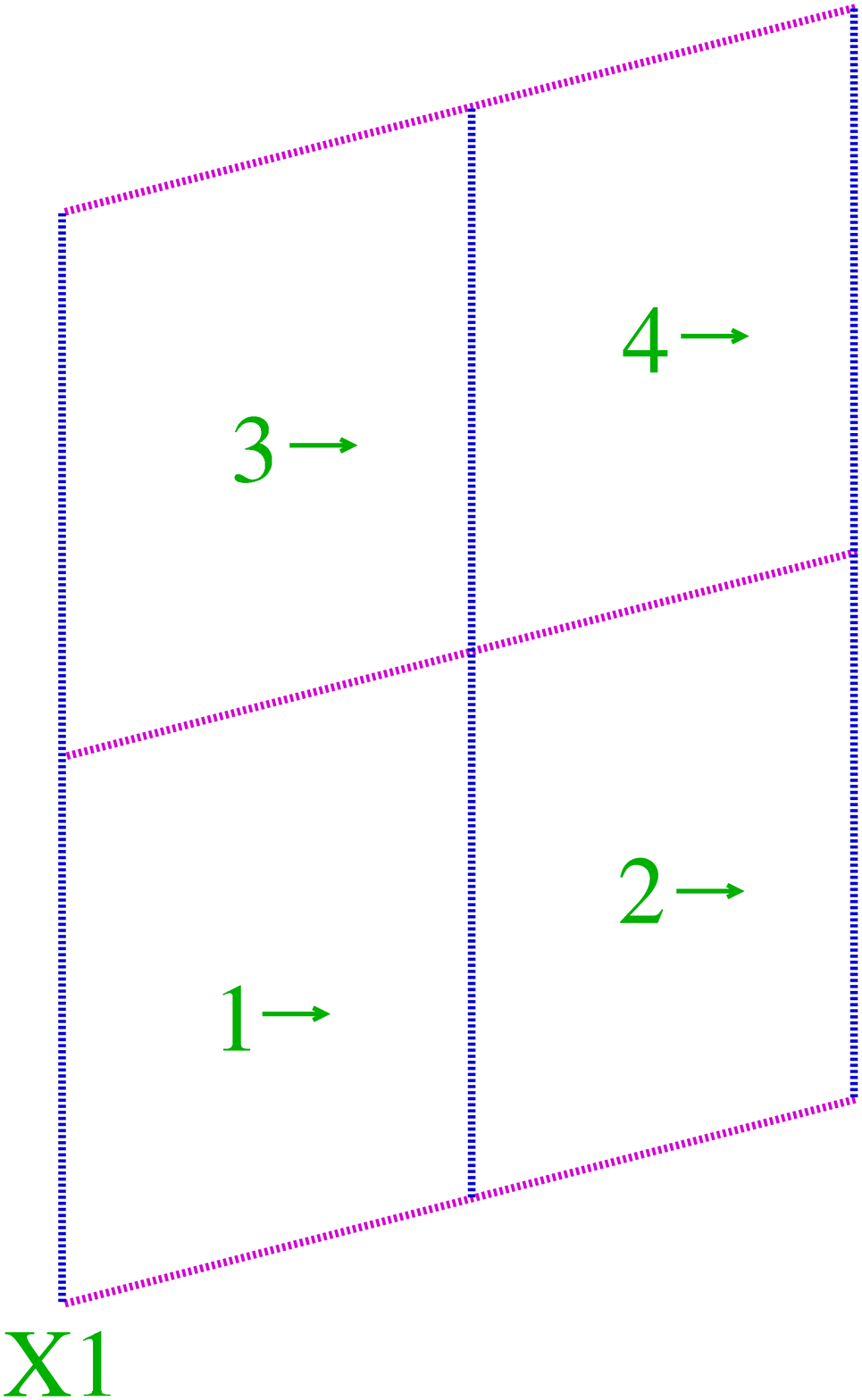} \hspace{5mm}
\includegraphics[width=.2\textwidth, keepaspectratio, clip]{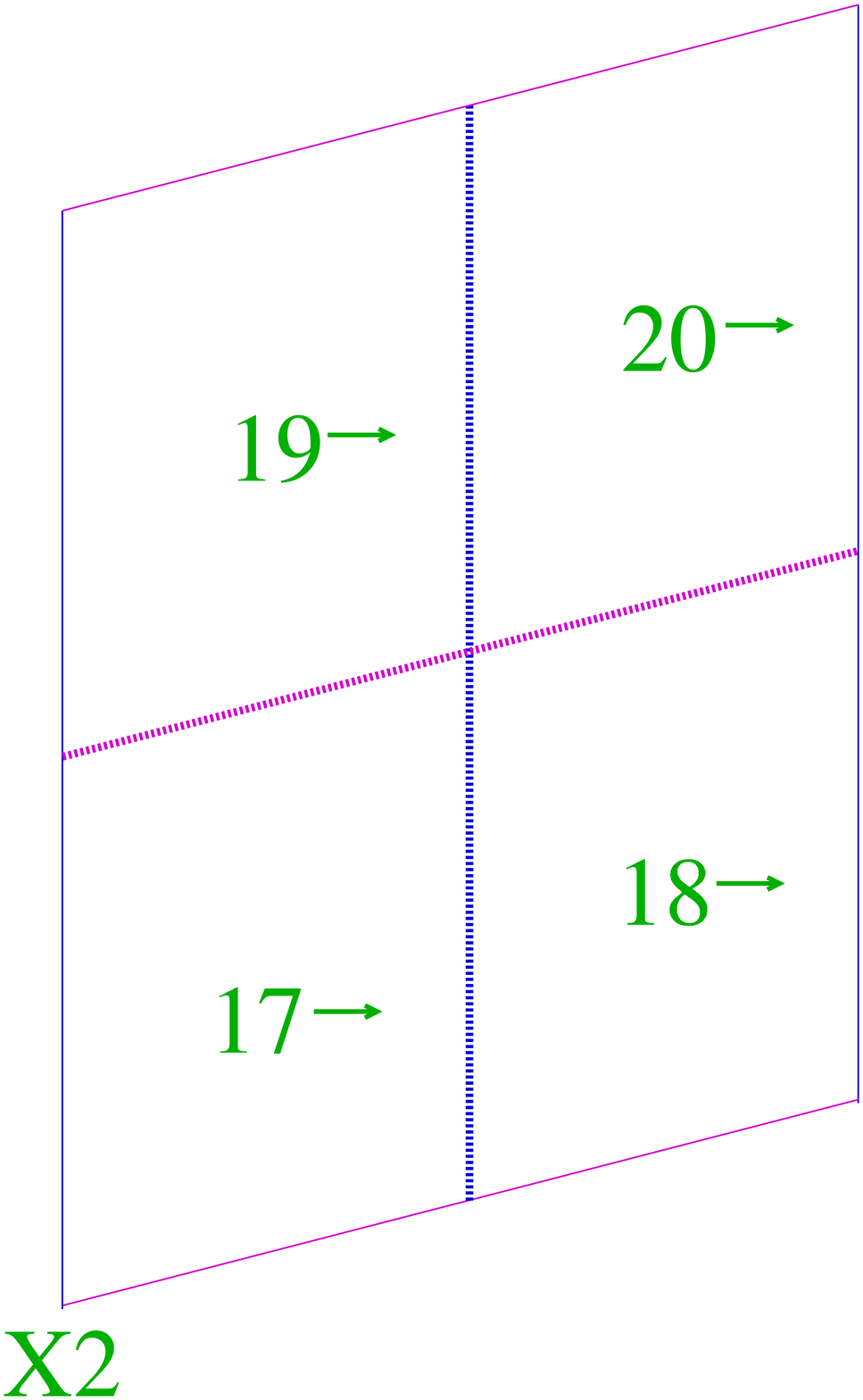}\\
\includegraphics[width=.25\textwidth, keepaspectratio, clip]{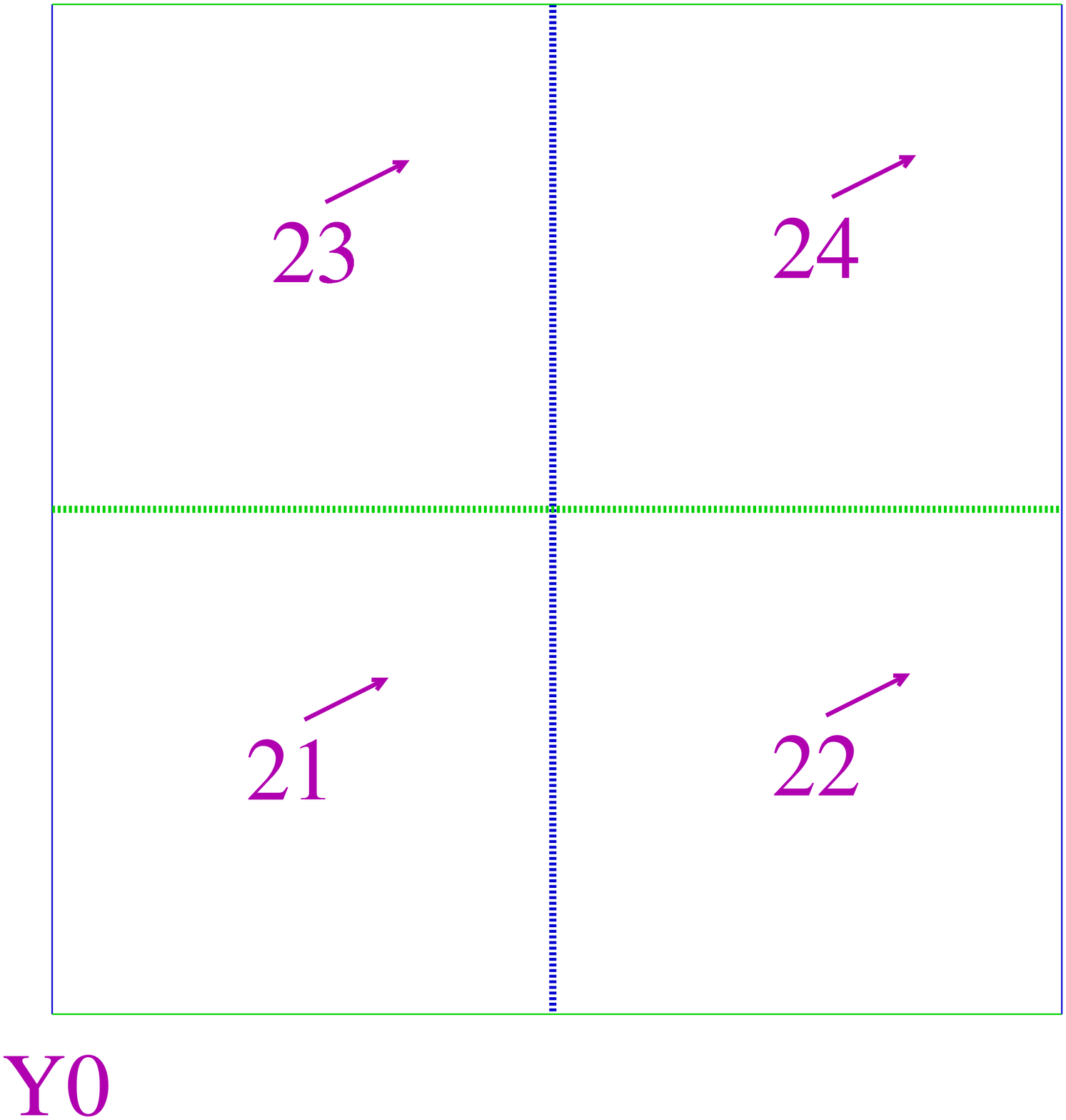} \hspace{5mm}
\includegraphics[width=.25\textwidth, keepaspectratio, clip]{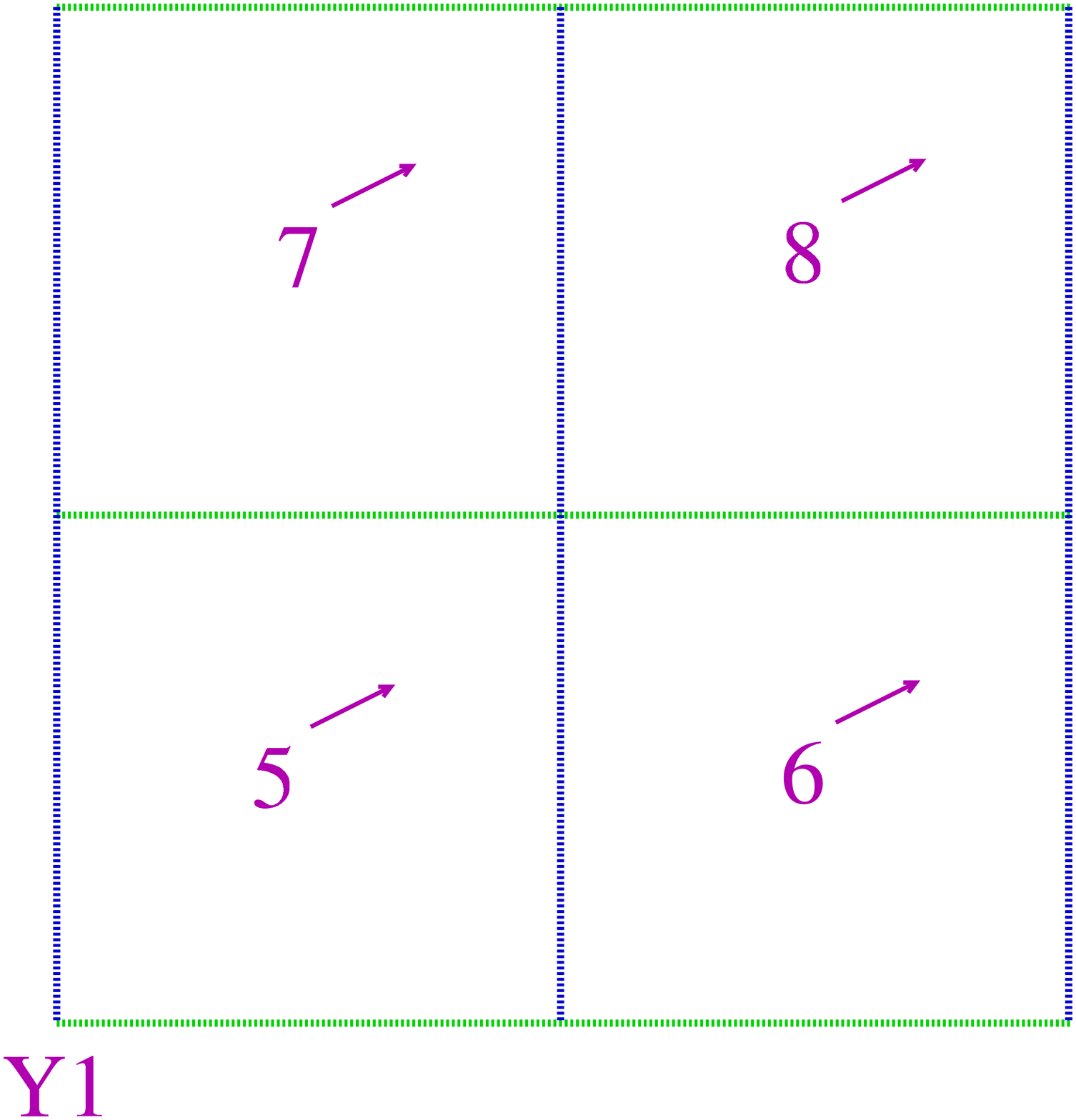} \hspace{5mm}
\includegraphics[width=.25\textwidth, keepaspectratio, clip]{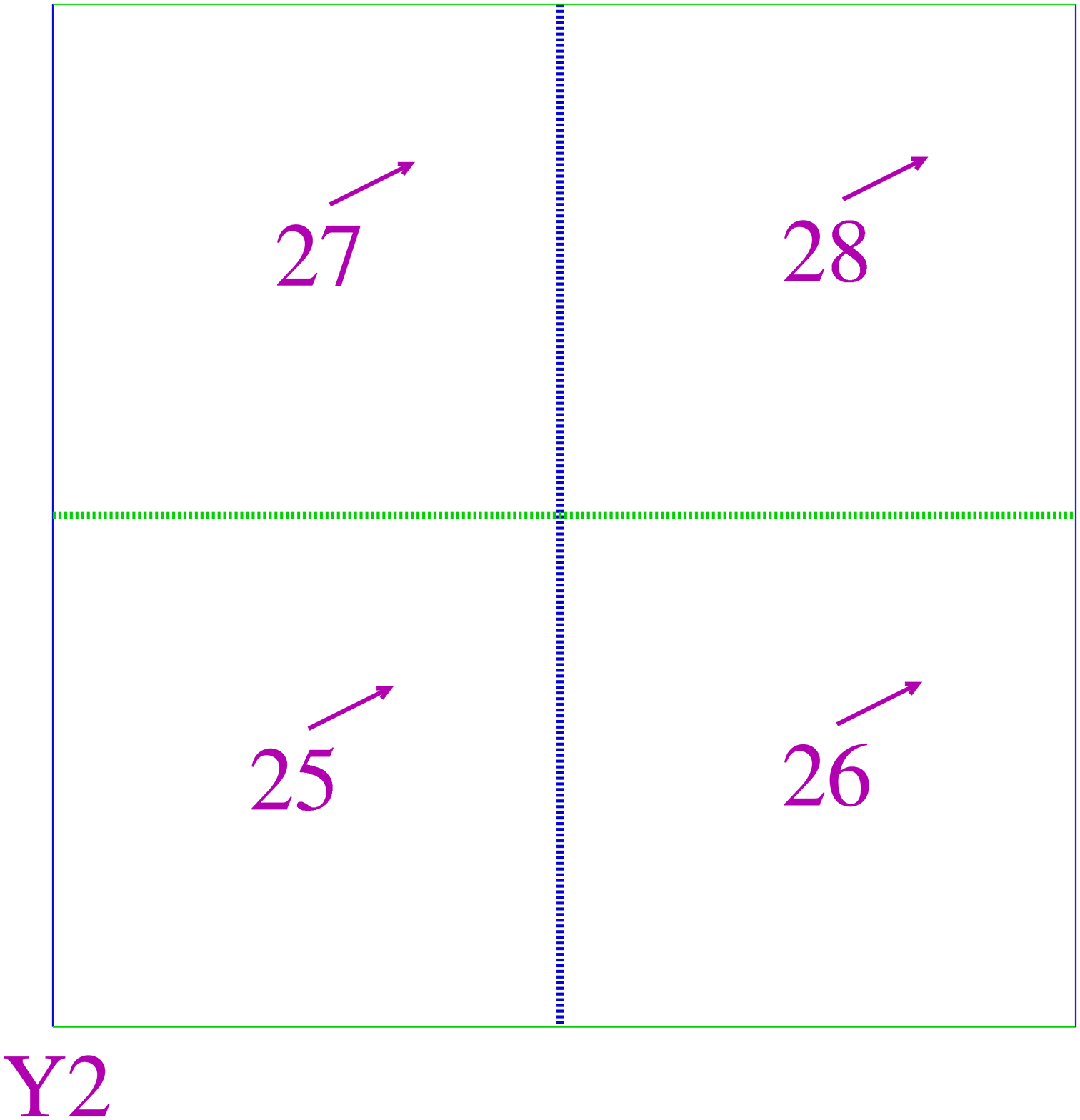}\\[2ex]
\includegraphics[width=.31\textwidth, keepaspectratio, clip]{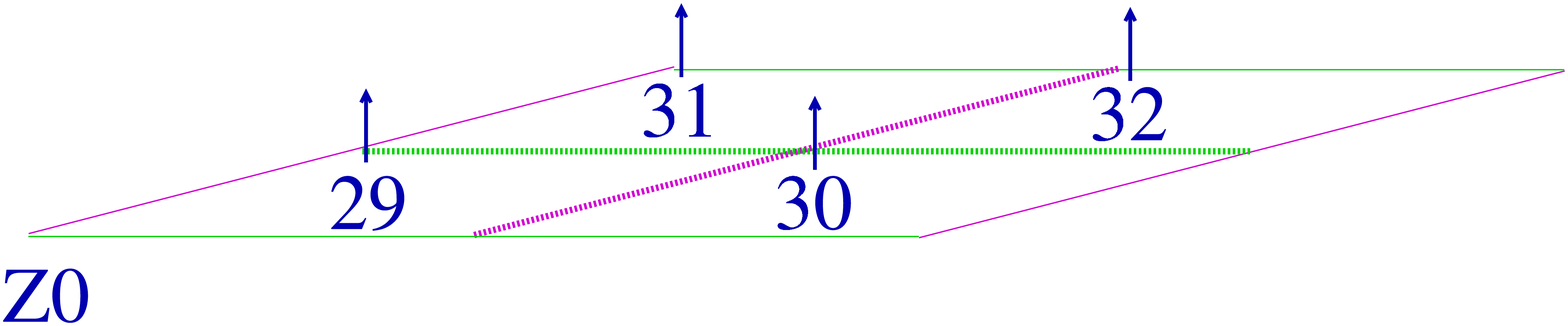} \hspace{3mm}
\includegraphics[width=.31\textwidth, keepaspectratio, clip]{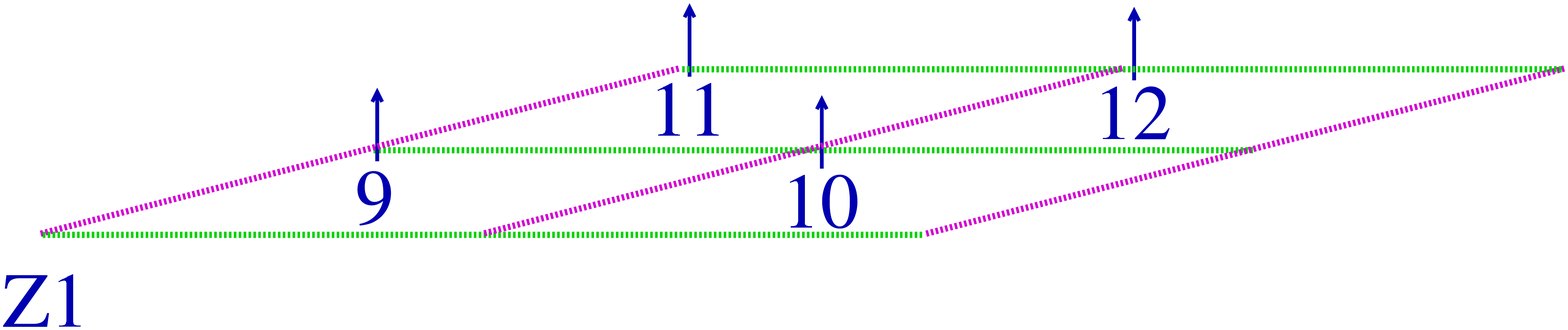} \hspace{3mm}
\includegraphics[width=.31\textwidth, keepaspectratio, clip]{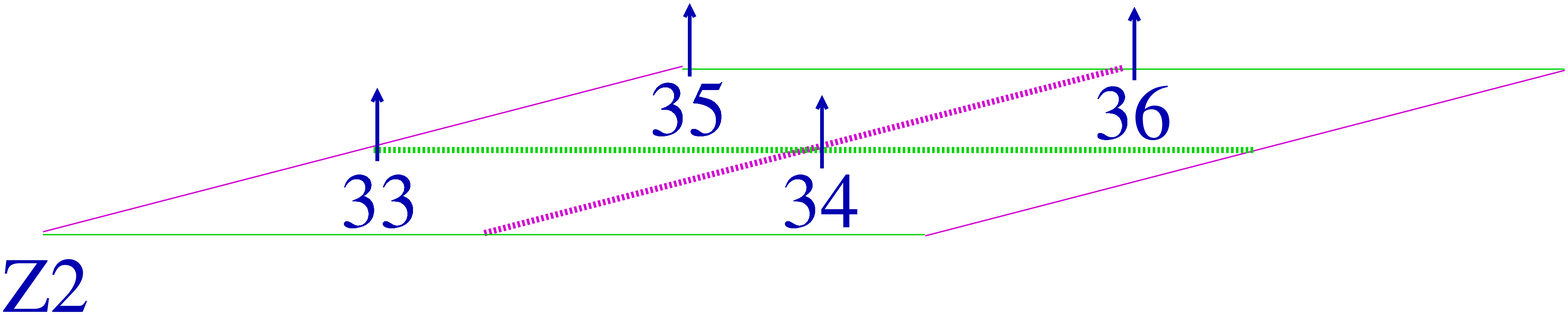}
\caption{Macro-element obtained after one regular mesh-refinement step}
\label{fig:MacroElm_Faces}
\end{figure}
For $\RTN$ subspace of $\Hdiv$, the particular two-level HB transformation, that induces this splitting, was introduced in the context of Rannacher-Turek elements for three-dimensional elliptic problems \cite{GKrausM-08b}.
Consider two consecutive discretizations $\mT_{H}$ (coarse level) and $\mTh$ (fine level). Figure~\ref{fig:MacroElm_Faces} illustrates a macro-element $G$ (at fine level) obtained from a coarse element by one regular mesh-refinement step. The colors green, magenta and blue represent the face directions and face DOFs for $x$, $y$ and $z$ directions, respectively.
Let $\varphi_{G} = \{\phi_i(x,y)\}_{i=1}^{36}$ be the macro-element vector of the nodal basis functions. Using the local numbering of DOF, as shown in Figure~\ref{fig:MacroElm_Faces} (second, third and fourth row of pictures), a macro-element level (local) transformation matrix $J_G$ is constructed based on differences and aggregates of basis functions $\phi_i$ and $\phi_j$ that correspond to a macro element face, i.e.,
\begin{subequations}\label{eq:J_G_RTN}
\begin{equation}
J_{G} = \frac{1}{4}
\left [
\begin{array}{rr}
4I & \\
& J_{G;22}
\end{array}
\right ],
\end{equation}
where $I$ is the $12 \times 12$ identity matrix and
\begin{equation}
J_{G;22} = \left [
\begin{array}{rrrrrr}
P_{D} & & & & & \\
& P_{D} & & & & \\
& & P_{D} & & & \\
& & & P_{D} & & \\
& & & & P_{D} & \\
& & & & & P_{D} \\
P_{A}^{1} & P_{A}^{2} & P_{A}^{3} & P_{A}^{4} & P_{A}^{5} & P_{A}^{6}
\end{array}
\right ].
\end{equation}
Here each block $P_{A}^{i}$, $i = 1, 2, \ldots, 6$, which reflects the basis functions obtained by aggregates, is a $6 \times 4$ matrix with all zeros except $i^{\mathrm{th}}$-row which has all ones. The block $P_{D}$, which reflects the orthogonal transformation to aggregates, and obtained by suitable combination of differences, is given by
\begin{equation}
P_{D} = \left [
\begin{array}{rrrr}
1 & -1 & 1 & -1 \\
1 & 1 & -1 & -1 \\
1 & -1 & -1 & 1
\end{array}
\right ].
\end{equation}
\end{subequations}
The transformations \eqref{eq:J_G_N}-\eqref{eq:J_G_RTN} define a two-level hierarchical basis $\hat{\varphi}_G$ locally, namely, $\hat{\varphi}_{G} = J_{G} \varphi_{G}$.
\subsection{Hierarchical splitting}
\label{sec:Splitting}
Let $A_G$ be the macro-element stiffness matrix corresponding to $G \in \mT = \mTh$. The global stiffness matrix $A_h$ can be written as
\[
A_h = \sum_{G\in\mT} R_G^T A_G R_G,
\]
where $R_G$ denotes the natural inclusion (canonical injection) of the matrix $A_G$ for all $G$ in $\mT$. Note that the matrix $A_{G}$ is of size $12 \times 12$ for two-dimensional $\Hcurl$ problem, and of size $36 \times 36$ for three-dimensional $\Hdiv$ problem.
Then the hierarchical two-level macro-element matrix is given by
\[
\hA_{G} = J_{G} A_{G} J_{G}^{T},
\]
and the related global two-level matrix can be obtained via assembling, i.e., $\hA_h = \sum_{G \in \mT} R_G^T \hA_G R_G$. Alternatively, one can compute the matrix $\hA_h$ via the triple matrix product
\begin{equation}\label{triple_mat_prod_3D}
\hA_h = J A_{h} J^{T} ,
\end{equation}
where the global transformation matrix $J$ is induced by the local transformations, i.e.,
\[
J\lvert_G = J_G, \quad \forall G \in \mT .
\]
In other words, global and local transformations are compatible in the sense that restricting $J$ to the DOF of any macro-element $G$ we obtain $J_G$.
Now, if we number those DOF that correspond to interior nodes of the macro elements first, the global two-level stiffness matrix $\hA_h$ has the $2 \times 2$ block structure
\begin{equation}\label{eq:hat_A}
\hA_h = \left[
\begin{array}{ll}
\hA_{11} & \hA_{12}  \\
\hA_{21} & \hA_{22}
\end{array}
\right],
\end{equation}
where $\hA_{11}$ corresponds to the \emph{interior unknowns}. We follow the \emph{first reduce} (FR) approach, see e.g., \cite{BlahetaMN-04, BlahetaMN-05, GKrausM-08, GKrausM-08b}, where these interior unknowns are first eliminated \emph{exactly}. This static condensation step can be written in the form
\begin{equation}\label{eq:hat_A_LU}
\hA_h =
\left[
\begin{array}{cc}
\hA_{11} & 0  \\
\hA_{21} & B
\end{array}
\right]
\left[
\begin{array}{cc}
I_1 & \hA_{11}^{-1} \hA_{12}  \\
0 & I_2
\end{array}
\right],
\end{equation}
with the Schur complement $B = \hA_{22} - \hA_{21} \hA_{11}^{-1} \hA_{12}$. Next, the matrix $B$ is partitioned into $2{\times}2$ blocks, i.e.,
\begin{equation}\label{eq:B2x2}
B = \left[
\begin{array}{ll}
B_{11} & B_{12}  \\
B_{21} & B_{22}
\end{array}
\right],
\end{equation}
where $B_{11}$ and $B_{22}$ correspond to the \emph{differences} and \emph{aggregates} of basis functions (associated with one macro-element edge or face), respectively. The matrix $B_{22}$ at level $\l$ then defines the coarse-grid matrix $A^{(\l-1)}$ in the AMLI hierarchy, cf.~(\ref{block_H}). This algorithm can be applied recursively on each level $\l=L,L-1,\ldots,1$. The resulting algorithm is then of optimal computational complexity, see e.g., \cite[Remark~3.1]{KrausTomar-11}.
\subsection{Local analysis}
\label{sec:LocalAnalysis}

In the two-level framework we denote by $\mV_{1}$ and $\mV_{2}$ the subspaces of the finite element space $\mV_h$. The space $\mV_{2}$ is spanned by the coarse-space basis functions (aggregates) and $\mV_{1}$ is the complement of $\mV_{2}$ in $\mV_h$, i.e., $\mV_h$ is a direct sum of $\mV_{1}$ and $\mV_{2}$:
\begin{equation}\label{Splitting:DS}
\mV_h = \mV_1 \oplus \mV_2.
\end{equation}
A measure for the quality of this splitting is the constant $\gamma$ in the strengthened CBS inequality, which is defined by the relation
\[
\gamma = \cos(\mV_{1},\, \mV_{2}) :=
\sup_{\begin{array}{l}\bm{u} \in \mV_{1},\;\bm{v}\in\mV_{2}\end{array}}
\displaystyle{\frac{\mH(\bm{u},\bm{v})}{\sqrt{\mH(\bm{u},\bm{u})\mH(\bm{v},\bm{v})}}} .
\]
It is well known (see, e.g., \cite{AxelssonG-83}) that $\gamma$ can be estimated locally over each macro element $G$, and that $\gamma = \max _{G} \gamma _{G}, $ where
\[
\gamma_{G} := \sup\limits_{\begin{array}{l}\bm{u}\in \mV_1(G),\; \bm{v}\in
\mV_{2}(G) \end{array}}
\displaystyle{\frac{\mH_{G}(\bm{u},\bm{v})}{\sqrt{\mH_{G}(\bm{u},\bm{u})\mH_{G}(\bm{v},\bm{v})}}} .
\]
The spaces $\mV_1(G)$, $\mV_2(G)$, and the bilinear form $\mH_{G}(\bm{u},\bm{v})$ correspond to the restriction of $\mV_1$, $\mV_2$, and $\mH(\bm{u},\bm{v})$, respectively, to the macro element $G$.
We perform this local analysis on the matrix level, where the splitting (\ref{Splitting:DS}) is obtained via the two-level hierarchical basis transformation described in Section~\ref{sec:HB}, and the space $\mV_h$ corresponds to the choice of lowest order Nedelec or Raviart-Thomas-Nedelec elements. In this setting the upper left block of $\hA_{h}$ is block-diagonal. Note that, for two-dimensional $\Hcurl$ problem, the diagonal blocks of $\hA_{11}$ are of size $4 \times 4$, which can be associated with the interior nodes $\{1, 2, \ldots, 4\}$ in the right picture of Figure~\ref{fig:MacroElm_Edges}, and for three-dimensional $\Hdiv$ problem, the diagonal blocks of $\hA_{11}$ are of size $12 \times 12$, which can be associated with the interior nodes $\{1, 2, \ldots, 12\}$ in the center column of second, third and fourth row of pictures in Figure~\ref{fig:MacroElm_Faces}.
Therefore, we first compute the local Schur complements arising from static condensation of the interior DOF and obtain the matrices $B_{G}$. Next we split each matrix $B_{G}$ as
\begin{equation*}
B_{G}=
\begin{bmatrix} B_{G,11} & B_{G,12} \\
B_{G,21} & B_{G,22}
\end{bmatrix}
\begin{array}{l}\} \mbox{ differences} \\
\} \mbox{ aggregates} \end{array},
\end{equation*}
written again in two-by-two block form. For two-dimensional $\Hcurl$ problem, the block $B_{G,11}$ and $B_{G,22}$ are both of size $4 \times 4$, and for three-dimensional $\Hdiv$ problem the block $B_{G,11}$ is of size $18 \times 18$ and the block $B_{G,22}$ is of size $6 \times 6$.
We have thus reduced the problem of estimating the CBS constant of the splitting (\ref{Splitting:DS}) to a small-sized local problem that involves the matrix $B_G$. Following the general theory, see \cite{AxelssonG-83, EijkhoutV-91}, to estimate the CBS constant $\gamma$, it suffices to compute the minimal eigenvalue of the generalized eigenproblem
\begin{equation}\label{LocGenEig}
S_G \textbf{v}_G = \lambda_{G,\min} B_{G,22} \textbf{v}_G, \quad \forall \textbf{v}_G,
\end{equation}
where $\, S_G = B_{G,22} - B_{G,21} B_{G,11}^{-1} B_{G,12}$. The CBS constant $\gamma$ can then be estimated as follows
\begin{equation}
\label{G_FR}
\gamma^2 \le \max_{G\in\mT}\gamma^2_G = \max_{G\in\mT} (1-\lambda_{G,\min}).
\end{equation}
Note that the matrix $B_{G,11}$ is a well conditioned matrix, see Figure~\ref{fig:cond_B11}, and therefore, it can be inverted cheaply, either by an iterative process or by, for example, an incomplete $LU$ factorization \cite{Saad-96}, which is denoted by $B_{11}^{i}$ in Figure~\ref{fig:cond_B11}.
\begin{figure}[!ht]
\centering
\begin{subfigure}[!ht]{.495\textwidth}
\centering
\includegraphics[width=.99\textwidth,keepaspectratio]{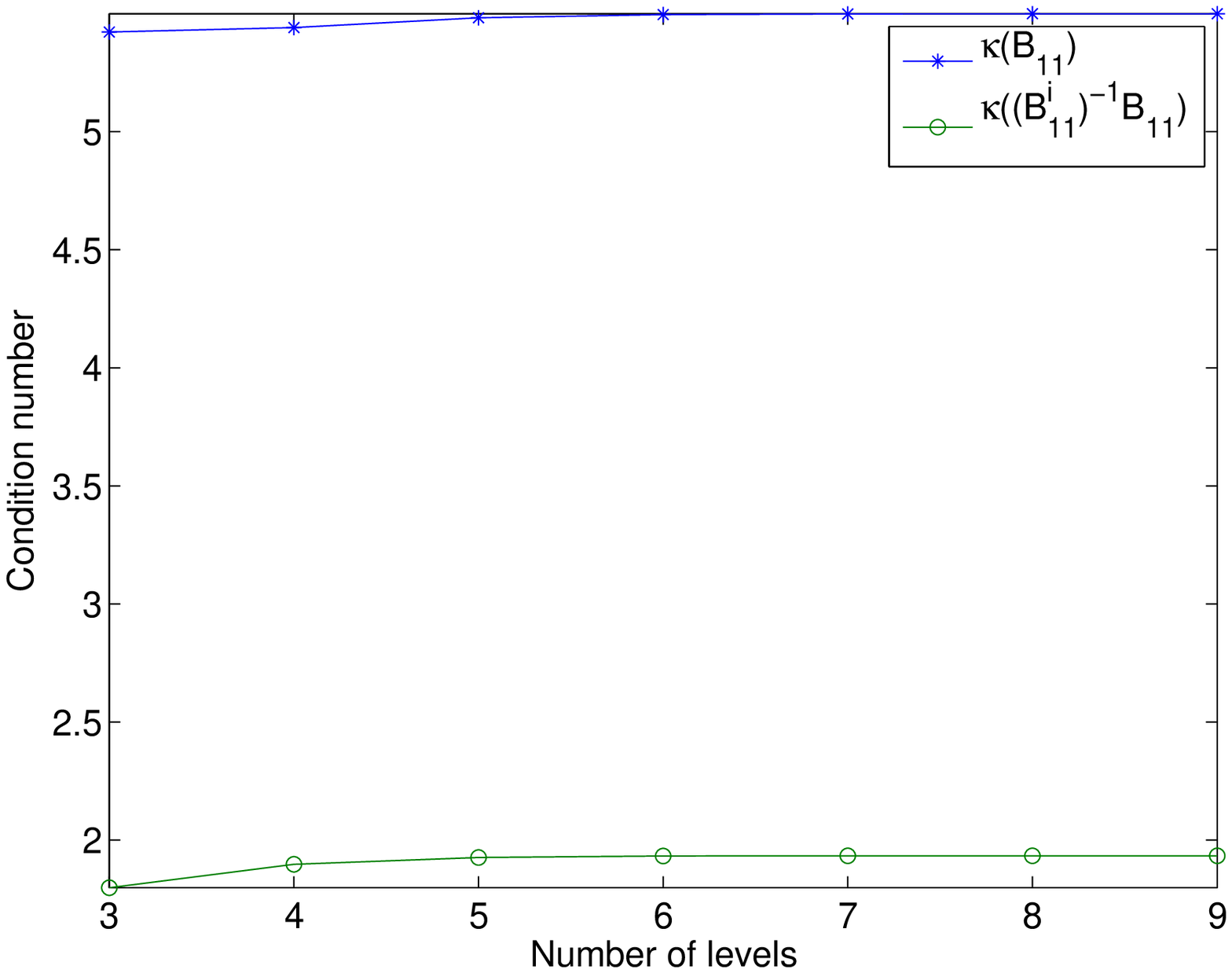}
\caption{Two-dimensional $\Hcurl$}
\end{subfigure}
\begin{subfigure}[!ht]{.495\textwidth}
\centering
\includegraphics[width=.99\textwidth,keepaspectratio]{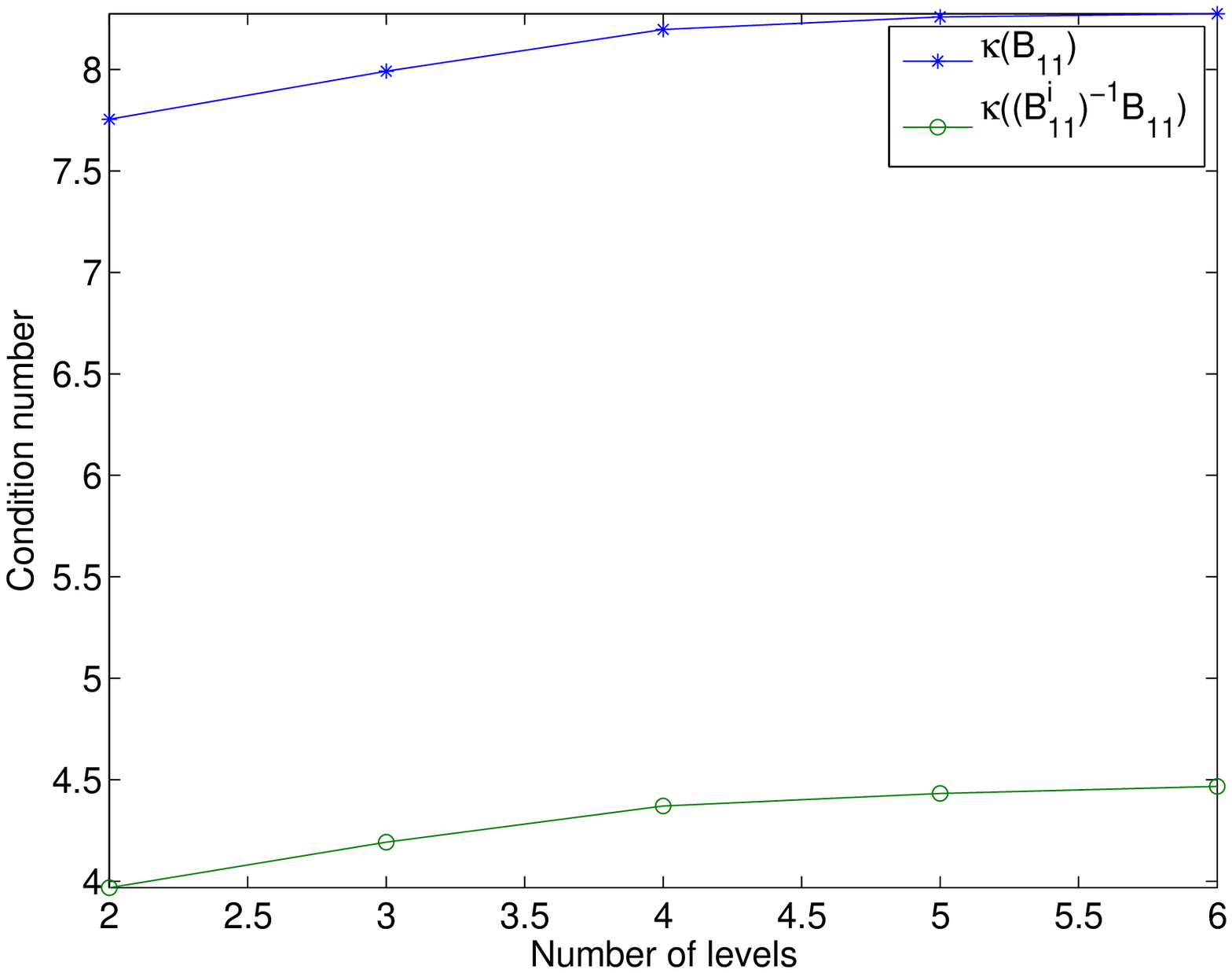}
\caption{Three-dimensional $\Hdiv$}
\end{subfigure}
\caption{Condition number of the matrix $B_{G,11}$}
\label{fig:cond_B11}
\end{figure}
We now first prove some auxiliary (stand-alone) results on algebraic sequences, which we will use to bound the CBS constant $\gamma$.
\begin{lemma}\label{lem:seq_abr}
For all $e>0$, consider the coupled sequences
\begin{subequations}
\label{eq:seq}
\begin{alignat}{2}
b_{0} & = e - 6, \qquad & a_{0} & = 2e + 6 = 2(b_{0} + 9), \label{eq:seq_0} \\
b_{\l+1} & = -b_{\l}^{2}/a_{\l}, \qquad & a_{\l+1} & = 2a_{\l} + b_{\l+1}, \quad \l = 0, 1, 2, \ldots . \label{eq:seq_m}
\end{alignat}
\end{subequations}
Let $r_{\l} = b_{\l}/a_{\l}$. Then we have
\begin{subequations}
\begin{alignat}{3}
& b_{\l+1}/a_{\l} = -r_{\l}^{2}, \quad & &
a_{\l+1}/a_{\l} = 2 - r_{\l}^{2}, \quad & &
r_{\l+1} = -r_{\l}^{2}/(2 - r_{\l}^{2}), \label{eq:rel_mp_m} \\
& a_{\l+1} - b_{\l+1} = 2a_{\l}, \quad & &
a_{\l+1} + b_{\l+1} = \dfrac{2}{a_{\l}}(a_{\l}^{2} - b_{\l}^{2}) = 2a_{\l}(1 - r_{\l}^{2}), \quad & &
\dfrac{a_{\l+1} + b_{\l+1}}{a_{\l+1} - b_{\l+1}} = 1 - r_{\l}^{2}. \label{eq:rel_amp_bmp}
\end{alignat}
\end{subequations}
Moreover, the following bound holds
\begin{subequations}
\label{eq:seq_ra}
\begin{align}
& -1 < r_{0} < 1/2, \mathrm{~and~} -1 < r_{\l} \le 0 \quad \forall ~\l = 1, 2, \ldots, \label{eq:bnd_rm} \\
& a_{\l} > \ldots a_{1} > a_{0} > 6, \quad
0 \le r_{\l}^{2} \le \ldots \le r_{1}^{2} \le r_{0}^{2} < 1, \forall \l = 0, 1, 2, \ldots. \label{eq:rel_rm_r0}
\end{align}
\end{subequations}
\end{lemma}
\begin{proof}
Using the definition of $r_{\l}$ in \eqref{eq:seq_m}, we get $b_{\l+1}/a_{\l} = -r_{\l}^{2}$, and thus $a_{\l+1}/a_{\l} = 2 - r_{\l}^{2}$. The last relation of \eqref{eq:rel_mp_m} then immediately follows.
The relations \eqref{eq:rel_amp_bmp} are also easily obtained from \eqref{eq:seq_m} and \eqref{eq:rel_mp_m}.
Clearly, for $e > 0$, we have $a_{0} > 6$, and since $r_{0} = b_{0}/a_{0} = (e-6)/(2e+6)$, it is easy to see that $-1 < r_{0} < 1/2$. The latter also implies that $0 \le r_{0}^{2} < 1$.
We now prove the remaining bounds using induction.
\begin{enumerate}
\item[$\l=0$.] Since $a_{1}/a_{0} = 2 - r_{0}^{2} > 1$, we have $a_{1} > a_{0} > 6$. 
Moreover, $r_{1} = -r_{0}^{2}/(2-r_{0}^{2})$. This implies that $-1 < r_{1} \le 0$, and thus $0 \le r_{1}^{2} < 1$. Furthermore, when $r_{0} \ne 0$, we have
\[
r_{1}^{2} = \left ( \dfrac{-r_{0}^{2}}{2-r_{0}^{2}} \right )^{2} \Rightarrow \dfrac{r_{1}^{2}}{r_{0}^{2}} = \dfrac{r_{0}^{2}}{(2 - r_{0}^{2})^{2}} < 1.
\]
And, since $r_{1} = 0$ if $r_{0} = 0$, we have $r_{1}^{2} \le r_{0}^{2} < 1$.
\item[$\l = n$.] Assume that the relations \eqref{eq:seq_ra} hold for $\l = n$. Since $a_{n+1}/a_{n} = 2 - r_{n}^{2} > 1$, we have $a_{n+1} > a_{n} > 6$.
Moreover, $r_{n+1} = -r_{n}^{2}/(2-r_{n}^{2})$. This implies that $-1 < r_{n+1} \le 0$, and thus $0 \le r_{n+1}^{2} < 1$. Also, when $r_{n} \ne 0$, we have
\[
r_{n+1}^{2} = \left ( \dfrac{-r_{n}^{2}}{2-r_{n}^{2}} \right )^{2} \Rightarrow \dfrac{r_{n+1}^{2}}{r_{n}^{2}} = \dfrac{r_{n}^{2}}{(2 - r_{n}^{2})^{2}} < 1.
\]
And, since $r_{n+1} = 0$ if $r_{n} = 0$, we have $r_{n+1}^{2} \le r_{n}^{2} < 1$.
\end{enumerate}
This concludes the proof. \hfill
\end{proof}
\begin{lemma}\label{lem:seq_c2}
Let $e > 0$ and the sequences $a_{\l}$ and $b_{\l}$ be as defined in Lemma~\ref{lem:seq_abr}. Then for
\begin{equation}
c_{\l, C}^{2} = \dfrac{36 (a_{\l} + b_{\l})}{(a_{\l}^{2} - 36)(a_{\l} - b_{\l})},
\label{eq:def_c2}
\end{equation}
the following bounds hold for all $\l = 0, 1, 2, \ldots$
\begin{equation}
c_{\l, C}^{2} < c_{\l-1, C}^{2} < \ldots < c_{1, C}^{2} < c_{0, C}^{2} <3/8.
\label{eq:bnd_seq_c2}
\end{equation}
\end{lemma}
\begin{proof}
From $a_{0} = 2e+6$ and $b_{0} = e-6$, we have $a_{0} - b_{0} = e+12$, $a_{0} + b_{0} = 3e$, $a_{0} - 6 = 2e$, and $a_{0} + 6 = 2(e+6)$. Substituting these relations in the definition of $c_{0, C}^{2}$, we get
\begin{equation}
c_{0, C}^{2} = \dfrac{27}{(e+6)(e+12)} < 3/8.
\label{eq:bnd_c2_0}
\end{equation}
Now
\begin{equation*}
c_{1, C}^{2} - c_{0, C}^{2} =
\dfrac{36\left ( (a_{1} + b_{1})(a_{0}^{2} - 36)(a_{0} - b_{0}) - (a_{0} + b_{0})(a_{1}^{2} - 36)(a_{1} - b_{1})\right )}
{(a_{1}^{2} - 36)(a_{1} - b_{1})(a_{0}^{2} - 36)(a_{0} - b_{0})}.
\end{equation*}
Substituting the values of $a_{0},a_{1},b_{0}$ and $b_{1}$, and after some lengthy, but simple calculations, we find that
\begin{equation*}
c_{1, C}^{2} - c_{0, C}^{2} =
\dfrac{108e \left ( -9e^{2} (312 + 80e + 5e^{2}) \right )}{(e+3)(a_{1}^{2} - 36)(a_{1} - b_{1})(a_{0}^{2} - 36)(a_{0} - b_{0})}.
\end{equation*}
Since the denominator is a positive quantity, we get $c_{1, C}^{2} - c_{0, C}^{2} < 0$, and thus
\begin{equation}
c_{1, C}^{2} < 3/8.
\label{eq:bnd_c2_1}
\end{equation}
For remaining bounds, we again use induction. Note that, using \eqref{eq:rel_amp_bmp} we get
\begin{equation}
c_{\l+1, C}^{2} = \dfrac{36 (a_{\l+1} + b_{\l+1})}{(a_{\l+1}^{2} - 36)(a_{\l+1} - b_{\l+1})} = \dfrac{36(1 - r_{\l}^{2})}{(a_{\l+1}^{2} - 36)}.
\label{eq:rel_gam2_rm_amp}
\end{equation}
Therefore, to show that $c_{\l+1, C}^{2} < 3/8$, it suffices to show that
\begin{equation}
a_{\l+1}^{2} - 36 > 96(1 - r_{\l}^{2}).
\label{eq:rel_amp_rm}
\end{equation}
Since $c_{1, C}^{2} < 3/8$, we clearly have $a_{1}^{2} - 36 > 96(1 - r_{0}^{2})$. Now assume that the relation \eqref{eq:rel_amp_rm} holds for $\l = n-1$, i.e.,
\begin{equation}
a_{n}^{2} - 36 > 96(1 - r_{n-1}^{2}).
\label{eq:rel_an_anm}
\end{equation}
Multiplying \eqref{eq:rel_an_anm} by $(2 - r_{n}^{2})^{2}$ and subtracting $36$ from both sides we get
\begin{align*}
(2 - r_{n}^{2})^{2} a_{n}^{2} - 36 & > 36(2 - r_{n}^{2})^{2} + 96(1 - r_{n-1}^{2})(2 - r_{n}^{2})^{2} - 36 \\
\Rightarrow \qquad a_{n+1}^{2} - 36 & > 96 \left ( (2 - r_{n}^{2})^{2} (11/8 - r_{n-1}^{2}) -3/8 \right ) .
\end{align*}
We need to show that $(2 - r_{n}^{2})^{2} (11/8 - r_{n-1}^{2}) -3/8 > 1 - r_{n}^{2}$, i.e.,
\begin{align}
g_{n} := (2 - r_{n}^{2})^{2} (11/8 - r_{n-1}^{2}) + r_{n}^{2} -11/8 > 0.
\label{eq:func_g}
\end{align}
From the recurrence relation on $r_{n}$ from \eqref{eq:rel_mp_m}, we have
\begin{equation*}
r_{n}^{2} = \dfrac{r_{n-1}^{4}}{(2-r_{n-1}^{2})^{2}}, \quad 2 - r_{n}^{2} = \dfrac{(r_{n-1}^{4} - 8r_{n-1}^{2} + 8)}{(2-r_{n-1}^{2})^{2}}.
\end{equation*}
Substituting these relations in $g_{n}$, and after some lengthy calculations we obtain
\begin{equation}
g_{n} = \dfrac{(1 - r_{n-1}^{2})^{2}}{(2 - r_{n-1}^{2})^{4}} ( -r_{n-1}^{6} + 15 r_{n-1}^{4} - 64 r_{n-1}^{2} + 66).
\label{eq:rel_g_rnm}
\end{equation}
Now for $r_{n-1}^{2} \in [0,1)$, we have
\[ 1 - r_{n-1}^{2} > 0, \quad 2 - r_{n-1}^{2} > 0, \quad 66 - 64 r_{n-1}^{2} > 0, \quad 15 r_{n-1}^{4} - r_{n-1}^{6} \ge 0, \]
which proves that $g_{n} > 0$, and that $a_{n+1}^{2} - 36 > 96 (1 - r_{n}^{2})$. Therefore, the inequality \eqref{eq:rel_amp_rm} holds for all $\l = 0, 1, \ldots $.
To prove the monotonicity of $c_{\l, C}^{2}$, we show that
\begin{equation}
f_{\l} := c_{\l+1, C}^{2}/c_{\l, C}^{2} < 1.
\label{eq:func_f}
\end{equation}
Using \eqref{eq:rel_gam2_rm_amp} we get
\begin{equation*}
f_{\l} = \dfrac{(1 - r_{\l}^{2}) (a_{\l}^{2} - 36)}{(1 - r_{\l-1}^{2})(a_{\l+1}^{2} - 36)}.
\end{equation*}
Multiplying numerator and denominator by $(2 - r_{\l}^{2})^{2}$, we obtain
\begin{align*}
f_{\l} & = \dfrac{(1 - r_{\l}^{2}) \left ( (2 - r_{\l}^{2})^{2} a_{\l}^{2} - 36 (2 - r_{\l}^{2})^{2}\right ) }
{(1 - r_{\l-1}^{2})(a_{\l+1}^{2} - 36) (2 - r_{\l}^{2})^{2}} \\
& = \dfrac{(1 - r_{\l}^{2})}{(1 - r_{\l-1}^{2})} \dfrac{\left ( a_{\l+1}^{2} - 36 + 36 ( 1 - (2 - r_{\l}^{2})^{2})\right )}
{(a_{\l+1}^{2} - 36) (2 - r_{\l}^{2})^{2}} \\
& = \dfrac{(1 - r_{\l}^{2})}{(1 - r_{\l-1}^{2})(2 - r_{\l}^{2})^{2}} + \dfrac{36(1 - r_{\l}^{2})( 1 - (2 - r_{\l}^{2})^{2})}
{(1 - r_{\l-1}^{2}) (a_{\l+1}^{2} - 36) (2 - r_{\l}^{2})^{2}}.
\end{align*}
Now since $c_{\l+1, C}^{2} < 3/8$, we have $(1 - r_{\l}^{2}) / (a_{\l+1}^{2} - 36) < 1/96$ from \eqref{eq:rel_amp_rm}. Therefore,
\begin{align*}
f_{\l} & < \dfrac{(1 - r_{\l}^{2})}{(1 - r_{\l-1}^{2})(2 - r_{\l}^{2})^{2}} +
\dfrac{36( 1 - (2 - r_{\l}^{2})^{2})} {96 (1 - r_{\l-1}^{2}) (2 - r_{\l}^{2})^{2}} \\
& = \dfrac{(1 - r_{\l}^{2}) + \dfrac{3}{8} ( 1 - (2 - r_{\l}^{2})^{2})} {(1 - r_{\l-1}^{2})(2 - r_{\l}^{2})^{2}}
= \dfrac{11/8 - r_{\l}^{2} - \dfrac{3}{8} (2 - r_{\l}^{2})^{2}} {(1 - r_{\l-1}^{2})(2 - r_{\l}^{2})^{2}}.
\end{align*}
This gives
\begin{align*}
f_{\l} - 1 & < \dfrac{11/8 - r_{\l}^{2} - \dfrac{3}{8} (2 - r_{\l}^{2})^{2} - (1 - r_{\l-1}^{2})(2 - r_{\l}^{2})^{2}}
{(1 - r_{\l-1}^{2})(2 - r_{\l}^{2})^{2}} \\
& = \dfrac{11/8 - r_{\l}^{2} + (2 - r_{\l}^{2})^{2} (- 11/8 + r_{\l-1}^{2})} {(1 - r_{\l-1}^{2})(2 - r_{\l}^{2})^{2}}.
\end{align*}
Using \eqref{eq:func_g} we therefore get
\begin{align*}
f_{\l} - 1 & < \dfrac{-g_{\l}} {(1 - r_{\l-1}^{2})(2 - r_{\l}^{2})^{2}} < 0,
\end{align*}
since $g_{\l} > 0$, $1 - r_{\l-1}^{2} > 0$, and $(2 - r_{\l}^{2})^{2} > 0$. This proves \eqref{eq:func_f} and concludes the proof.
\end{proof}
\begin{lemma}\label{lem:seq_c2_3D}
Let $e > 0$ and the sequences $a_{\l}$ and $b_{\l}$ be as defined in Lemma~\ref{lem:seq_abr}. Then for
\begin{equation}
c_{\l, D}^{2} = \dfrac{72 (a_{\l} + b_{\l})}{(a_{\l} + 12)(a_{\l} - 6)(a_{\l} - b_{\l})},
\label{eq:def_c2_3D}
\end{equation}
the following bounds hold for all $\l = 0, 1, 2, \ldots$
\begin{equation}
c_{\l, D}^{2} < c_{\l-1, D}^{2} < \ldots < c_{1, D}^{2} < c_{0, D}^{2} <1/2.
\label{eq:bnd_seq_c2_3D}
\end{equation}
\end{lemma}
\begin{proof}
Substituting the relations for $a_{0}, b_{0}, a_{0} - b_{0}, a_{0} + b_{0}, a_{0} - 6 $, and $a_{0} + 6 $ from Lemma~\ref{lem:seq_c2} in the definition of $c_{0, D}^{2}$, we get
\begin{equation}
c_{0, D}^{2} = \dfrac{54}{(e+9)(e+12)} < 1/2.
\label{eq:bnd_c2_0_3D}
\end{equation}
Now substituting the values of $a_{0},a_{1},b_{0}$ and $b_{1}$, and after some lengthy, but simple calculations, we find that
\begin{equation}
c_{1, D}^{2} - c_{0, D}^{2} =
\dfrac{-486e (5e^{2} + 88e + 372)}{(e + 9)(e + 12)(7e + 48)(7e^{2} + 84e + 108)} < 0.
\label{eq:bnd_c2_1_3D}
\end{equation}
For remaining bounds, we use induction and proceed as follows. Let $t_{m} := 1/2 - c_{m, D}^{2}$ and $t_{m+1} := 1/2 - c_{m+1, D}^{2}$. Then, expanding $a_{m+1}$ and $b_{m+1}$ in terms of $a_{m}$ and $b_{m}$, and dropping the subscripts of $a_{m}$ and $b_{m}$ for brevity reasons, we get
\begin{subequations}
\begin{align}
t_{m} := 1/2 - c_{m, D}^{2} & =
\dfrac{-216 a + 6 a^{2} + a^{3} - 72 b - 6 a b - a^{2} b}{2 (a - 6) (a + 12) (a-b)} =: \dfrac{n_{m}}{d_{m}},
\label{eq:rel_tn} \\
t_{m+1} := 1/2 - c_{m+1, D}^{2} & =
\dfrac{-216 a^{2} + 12 a^{3} + 4 a^{4} + 144 {b}^2 - 6 a b^{2} - 4 a^{2} b^{2} + b^{4}}
{2 (-6 a + 2 a^{2} - b^{2}) (12 a + 2 {a}^2-b^{2})} =: \dfrac{n_{m+1}}{d_{m+1}},
\label{eq:rel_tnp}
\end{align}
\end{subequations}
where $n_{m}$ and $n_{m+1}$ are the numerators of $t_{m}$ and $t_{m+1}$, respectively, and $d_{m}$ and $d_{m+1}$ are the denominators of $t_{m}$ and $t_{m+1}$, respectively. Assume that the relation \eqref{eq:bnd_seq_c2_3D} holds for $\l = m \ge 1$, i.e., $t_{m} > 0$. We need to show that $t_{m+1} > 0$. Since $a > 6$, $a > |b|$, and $b < 0$ for $m \ge 1$, we see that $d_{m}$ and $d_{m+1}$ are positive. Therefore, it suffices to show that $n_{m+1}$ is positive whenever $n_{m}$ is positive. Given $a/2 > 1$, we consider $n_{m+1} - \dfrac{a}{2} n_{m}$. We have
\begin{align*}
2(n_{m+1} - \dfrac{a}{2} n_{m}) & = (a + b)(7 a^{3} + 18 a^{2} - 6 a ^{2} b + 288 b + 2 b^{3} - 216 a - 12 ab - 2 a b^{2})\\
& = (a + b)(-6b (a^{2} + 2 a - 48) + 3a (a^{2} + 6 a - 72) + 2 (a^{3} + b^{3}) + 2 a (a^{2} - b^{2}))\\
& > 0,
\end{align*}
since $a > 6$, $a > |b|$, and $b < 0$ for $m \ge 1$. This proves that $n_{m+1} > 0$, and hence, $t_{m+1} > 0$.
The monotonicity of $c_{\l, D}^{2}$ can be shown by using \eqref{eq:bnd_c2_1_3D} and showing the induction that $c_{m+1, D}^{2} - c_{m, D}^{2} < 0$ whenever $c_{m, D}^{2} - c_{m-1, D}^{2} < 0$. The details are omitted here (the results can also be verified by using algebraic cylindrical decomposition in a computer algebra system like Mathematica \cite{mathematica}).
\end{proof}
The sequences $a_{\l}$, $b_{\l}$, and $r_{\l}$ are plotted in Figure~\ref{fig:seq_abr}, and the sequences $c_{\l, C}^{2}$ and $c_{\l, D}^{2}$ are plotted in Figure~\ref{fig:seq_c}.
\begin{figure}[!ht]
\centering
\includegraphics[width=.3\textwidth,keepaspectratio]{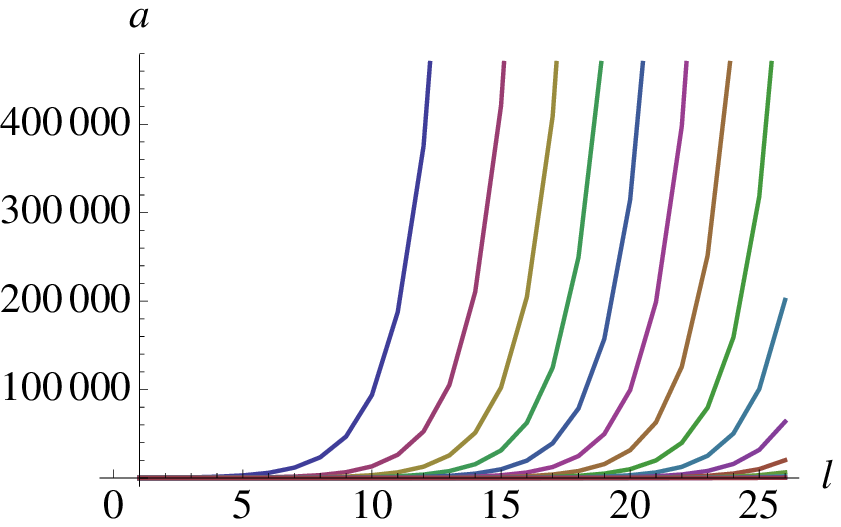}
\quad
\includegraphics[width=.3\textwidth,keepaspectratio]{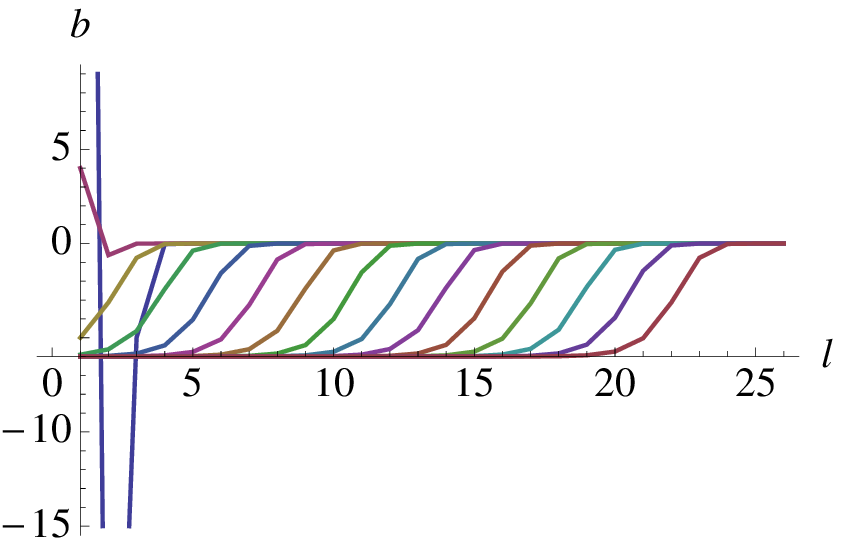}
\quad
\includegraphics[width=.3\textwidth,keepaspectratio]{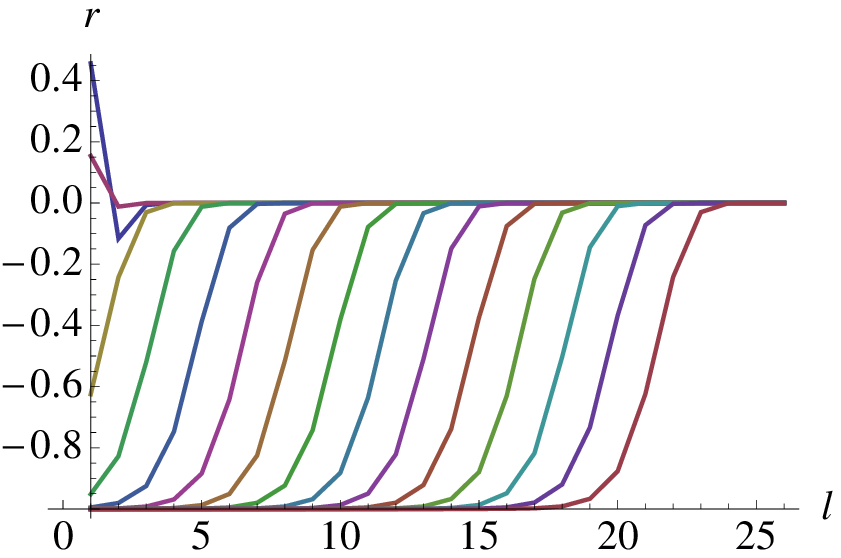}
\caption{$a_{\l}$, $b_{\l}$, and $r_{\l}$ for $e = 10^{m_{0}}$, where $m_{0} = \{2,1,0, \ldots -11,-12\}$ (left to right)}
\label{fig:seq_abr}
\end{figure}
\begin{figure}[!ht]
\centering
\begin{subfigure}{.495\textwidth}
\centering
\includegraphics[width=.99\textwidth,keepaspectratio]{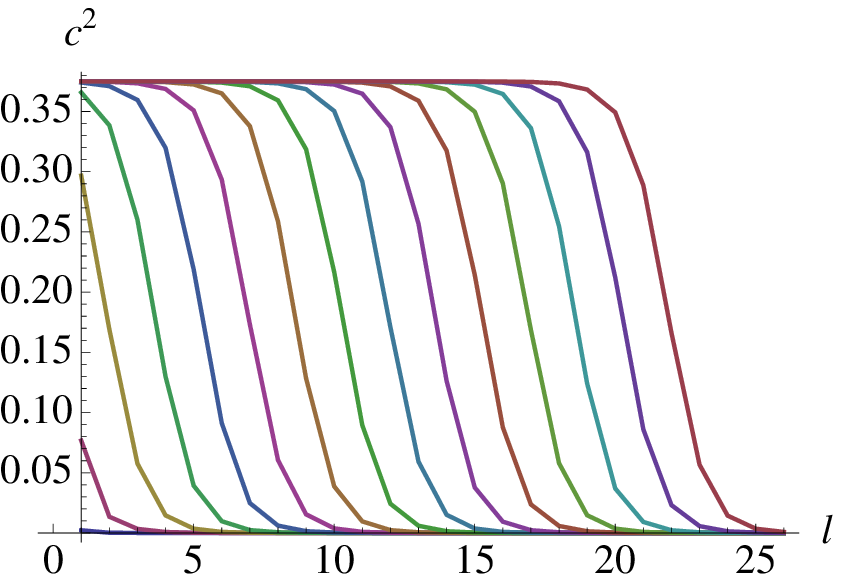}
\caption{$c_{\l, C}^{2}$}
\end{subfigure}
\begin{subfigure}{.495\textwidth}
\centering
\includegraphics[width=.99\textwidth,keepaspectratio]{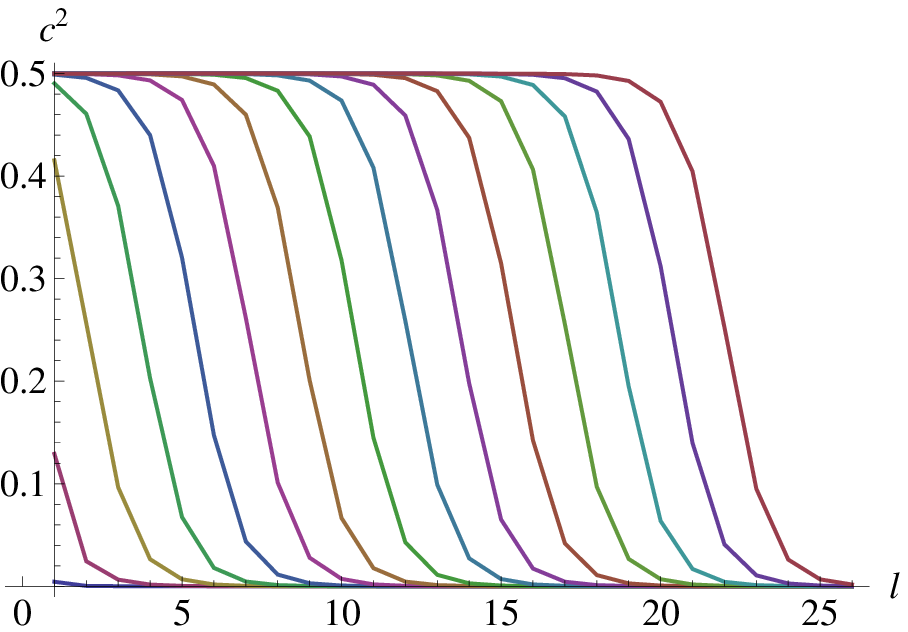}
\caption{$c_{\l, D}^{2}$}
\end{subfigure}
\caption{$c_{\l}^{2}$ for $e = 10^{m_{0}}$, where $m_{0} = \{2,1,0, \ldots -11,-12\}$ (left to right)}
\label{fig:seq_c}
\end{figure}
We are now in a position to prove the following theorem which provides a theoretical estimate that holds on all levels of recursive splitting of the $\N$ subspace of $\Hcurl$, and $\RTN$ subspace of $\Hdiv$.
\begin{theorem}\label{thm:Gamma}
Consider the bilinear form \eqref{eq:ModProb}, where $0 < \alpha, \beta < \infty$, and the related discrete problem (\ref{eq:ModProbDiscrete}) on the $\N$ subspace of two-dimensional $\Hcurl$ or $\RTN$ subspace of three-dimensional $\Hdiv$. Assuming that the underlying mesh is uniform, the CBS constant $\gamma$ related to the hierarchical splitting (\ref{Splitting:DS}) has the upper bound ${\displaystyle \gamma \leq \gamma_{G} < \sqrt{\Theta}}$, where $\Theta$ is $3/8$ for two-dimensional $\Hcurl$ problem, and $1/2$ for three-dimensional $\Hdiv$ problem. This upper bound holds for each step of the recursive hierarchical splitting. Moreover, $\gamma^{(L - \l)}$ is monotonically strictly decreasing and has an upper bound of $\sqrt{\Theta}$ for all $\l = 0, 1, \ldots, L$, i.e.,
\begin{align}
\gamma^{(0)} < \gamma^{(1)} < \ldots < \gamma^{(\l)} < \ldots < \gamma^{(L-1)} < \gamma^{(L)} < \sqrt{\Theta}.
\label{eq:bnd_GammaL}
\end{align}
\end{theorem}
\begin{proof}
In order to prove this uniform bound for $\gamma$ we study the generalized eigenproblem (\ref{LocGenEig}). At level $L$ of the finest discretization the macro-element matrix $\hA_G$, which is the same for all $G$ in $\mT_{h_L}$ for a uniform mesh, can be represented in the form
\begin{equation}\label{Rep_AGh_0}
\hA_G^{(L)} = J_G \left( \sum_{K \in G \subset \mT_{h_\ell}} R_K^T A_K^{(L)} R_K \right) J_G^T.
\end{equation}
We first focus on two-dimensional $\Hcurl$ problem, for which
\begin{align}
A_{K, C}^{(L)} = \frac{\beta}{6 h^{2}} \left [
\begin{array}{rrrr}
a_{0} & b_{0} & -6 & 6 \\
b_{0} & a_{0} & 6 & -6 \\
-6 & 6 & a_{0} & b_{0} \\
6 & -6 & b_{0} & a_{0} \\
\end{array}
\right ], \quad \forall K \in G, ~ \forall G \subset \mT_{h_L}.
\label{eq:ElmMatAL}
\end{align}
The variables $a_{0}$ and $b_{0}$ are defined in Lemma~\ref{lem:seq_abr}, $e$ and $\kappa$ are defined before \eqref{eq:ElmMatAk_k_2D}, and the local transformation matrix $J_{G}$ is defined according to \eqref{eq:J_G_N}.
The lower-right $4 \times 4$ block of the matrix $B_G$ and the Schur complement $S_G$ for the first splitting (at level $L$) are to be found
\begin{subequations}\label{eq:BG22_SG_L}
\begin{align}
B_{G,22}^{(L)} & = \frac{\beta}{6 h^{2}} \left [
\begin{array}{rrrr}
p_0 & q_0 & -3/2 & 3/2 \\
q_0 & p_0 & 3/2 & -3/2 \\
-3/2 & 3/2 & p_0 & q_0 \\
3/2 & -3/2 & q_0 & p_0
\end{array}
\right ],\\
S_{G}^{(L)} & =
\frac{\beta}{6 h^{2}} \left [
\begin{array}{rrrr}
s_0 & t_0 & -3/2 & 3/2 \\
t_0 & s_0 & 3/2 & -3/2 \\
-3/2 & 3/2 & s_0 & t_0 \\
3/2 & -3/2 & t_0 & s_0
\end{array}
\right ].
\end{align}
\end{subequations}
with
\begin{alignat*}{2}
q_{0} & = -b_{0}^{2}/4a_{0}, \quad
& p_{0} & = a_{0}/2 + q_{0} ,\\
t_{0} & = \dfrac{36a_{0} + 72b_{0} + a_{0}b_{0}^{2}}{144 - 4a_{0}^{2}}, \quad
& s_{0} & =  a_{0}/2 + t_{0}.
\end{alignat*}
The generalized eigenproblem (\ref{LocGenEig}) has two different two-fold eigenvalues, namely $\lambda_{1,2} = 1$ and
\[
\lambda_{3,4} = \dfrac{a_{0}(a_{0}^{2} - a_{0}b_{0} - 72)}{(a_{0}^{2} - 36)(a_{0} - b_{0})},
\]
which shows that
\begin{equation}\label{gamma_G_0}
\left(\gamma_G^{(L)}\right)^{2} \le 1 - \lambda_{3,4} = \dfrac{36(a_{0} + b_{0})}{(a_{0}^{2} - 36)(a_{0} - b_{0})} .
\end{equation}
Note that the coefficient $\beta$ does not appear in the bound for $\gamma$ since the factor $\frac{\beta}{6 h^{2}}$ appear in both the matrices of the generalized eigenproblem (\ref{LocGenEig}), and thus does not affect the eigenvalues.
Now in order to compute a similar bound for the second splitting (at level $L-1$) we have to use the relation $A_K^{(L-1)} := B_{G,22}^{(L)}$. In general, for the $(\l+1)^{\mathrm{th}}$ splitting (at level $L-\l$) the relation
\begin{equation}\label{AK_BG}
A_K^{(L - \l)} := B_{G,22}^{(L - \l + 1)}
\end{equation}
is to be used in the assembly of $\hA_G^{L - \l}$, i.e.,
\begin{equation}\label{eq:Ass_AGh_lk}
\hA_G^{(L - \l)} = J_G \left( \sum_{K \in G \subset \mT_{h_{L - \l}}} R_K^T A_K^{(L - \l)} R_K \right) J_G^T .
\end{equation}
Repeating the computations, we find that the relation \eqref{eq:Ass_AGh_lk} holds for all levels $\l = 1, 2, \ldots, L-1, L$, and the element stiffness matrix $A_{K}^{L - \l}$ (after $\l$ coarsening steps) is given by
\begin{align}
A_{K, C}^{(L - \l)} =
\frac{\beta}{6 (2^{\l} h)^{2}}
\left [
\begin{array}{rrrr}
a_{\l} & b_{\l} & -6 & 6 \\
b_{\l} & a_{\l} & 6 & -6 \\
-6 & 6 & a_{\l} & b_{\l} \\
6 & -6 & b_{\l} & a_{\l} \\
\end{array}
\right ], \quad \forall K \in G, ~ \forall G \subset \mT_{h_{L - \l}} ,
\label{eq:ElmMatAl}
\end{align}
where the sequences $a_{\l}$ and $b_{\l}$ are defined in \eqref{eq:seq}. Thus, the bound for $\gamma_G$ at level $L-\l$ reads
\begin{align}
\bigl(\gamma^{(L-\l)}_{G}\bigr)^{2} = \dfrac{36(a_{\l} + b_{\l})}{(a_{\l}^{2} - 36)(a_{\l} - b_{\l})}.
\label{eq:def_GammaL}
\end{align}
The result \eqref{eq:bnd_GammaL} then follows by taking $\gamma^{L-\l}_{G} = c_{\l, C}$, where $c_{\l, C}$ is defined in Lemma~\ref{lem:seq_c2}.
For three-dimensional $\Hdiv$ problem we have
\begin{align}
A_{K, D}^{(L)} = \frac{\beta}{6 h^{3}} \left [
\begin{array}{rrrrrr}
a_{0} & b_{0} & 6 & -6 & 6 & -6 \\
b_{0} & a_{0} & -6 & 6 & -6 & 6 \\
6 & -6 & a_{0} & b_{0} & 6 & -6 \\
-6 & 6 & b_{0} & a_{0} & -6 & 6 \\
6 & -6 & 6 & -6 & a_{0} & b_{0} \\
-6 & 6 & -6 & 6 & b_{0} & a_{0}
\end{array}
\right ], \quad \forall K \in G, ~ \forall G \subset \mT_{h_L},
\label{eq:ElmMatAL_3D}
\end{align}
and the local transformation matrix $J_{G}$ is defined according to \eqref{eq:J_G_RTN}.
The lower-right $6 \times 6$ block of the matrix $B_G$ and the Schur complement $S_G$ for the first splitting (at level $L$) are to be found (using e.g., Mathematica \cite{mathematica})
\begin{subequations}\label{eq:BG22_SG_L_3D}
\begin{align}
B_{G,22}^{(L)} & = \frac{\beta}{6 h^{3}} \left [
\begin{array}{rrrrrr}
p_{0} & q_{0} & 3/4 & -3/4 & 3/4 & -3/4 \\
q_{0} & p_{0} & -3/4 & 3/4 & -3/4 & 3/4 \\
3/4 & -3/4 & p_{0} & q_{0} & 3/4 & -3/4 \\
-3/4 & 3/4 & q_{0} & p_{0} & -3/4 & 3/4 \\
3/4 & -3/4 & 3/4 & -3/4 & p_{0} & q_{0} \\
-3/4 & 3/4 & -3/4 & 3/4 & q_{0} & p_{0}
\end{array}
\right ], \\
S_{G}^{(L)} & = \frac{\beta}{6 h^{3}} \left [
\begin{array}{rrrrrr}
s_{0} & t_{0} & 3/4 & -3/4 & 3/4 & -3/4 \\
t_{0} & s_{0} & -3/4 & 3/4 & -3/4 & 3/4 \\
3/4 & -3/4 & s_{0} & t_{0} & 3/4 & -3/4 \\
-3/4 & 3/4 & t_{0} & s_{0} & -3/4 & 3/4 \\
3/4 & -3/4 & 3/4 & -3/4 & s_{0} & t_{0} \\
-3/4 & 3/4 & -3/4 & 3/4 & t_{0} & s_{0}
\end{array}
\right ],
\end{align}
\end{subequations}
with
\begin{alignat*}{2}
q_{0} & = -b_{0}^{2}/8a_{0} , \quad & p_{0} & = a_{0}/4 + q_{0} , \\
t_{0} & = \dfrac{-72 a_{0} - 144 b_{0} - 6 b_{0}^{2} - a_{0} b_{0}^{2}}{8(a_{0} - 6)(a_{0} + 12)} ,
\quad
& s_{0} & = a_{0}/4 + t_{0}.
\end{alignat*}
The generalized eigenproblem (\ref{LocGenEig}) has two different three-fold eigenvalues, namely $\lambda_{1, 2, 3} = 1$ and
\[
\lambda_{4, 5, 6} = \dfrac{a_{0}(a_{0}^{2} - a_{0}b_{0} + 6 a_{0} - 6 b_{0} - 144)}{(a_{0} + 12)(a_{0} - 6)(a_{0} - b_{0})},
\]
which shows that
\begin{equation}\label{gamma_G_0_3D}
\left(\gamma_G^{(L)}\right)^{2} \le 1 - \lambda_{4, 5, 6} = \dfrac{72(a_{0} + b_{0})}{(a_{0} + 12)(a_{0} - 6)(a_{0} - b_{0})} .
\end{equation}
As before, to compute a similar bound for the $(\l+1)^{\mathrm{th}}$ splitting the relation \eqref{AK_BG} is to be used in the assembly of $\hA_G^{L - \l}$, see \eqref{eq:Ass_AGh_lk}.
Repeating the computations, we find that the relation \eqref{eq:Ass_AGh_lk} holds for all levels $\l = 1, 2, \ldots, L-1, L$, and the element stiffness matrix $A_{K}^{L - \l}$ (after $\l$ coarsening steps) is given by
\begin{align}
A_{K, D}^{(L - \l)} =
\frac{\beta}{6 (2^{\l} h)^{3}}
\left [
\begin{array}{rrrrrr}
a_{\l} & b_{\l} & 6 & -6 & 6 & -6 \\
b_{\l} & a_{\l} & -6 & 6 & -6 & 6 \\
6 & -6 & a_{\l} & b_{\l} & 6 & -6 \\
-6 & 6 & b_{\l} & a_{\l} & -6 & 6 \\
6 & -6 & 6 & -6 & a_{\l} & b_{\l} \\
-6 & 6 & -6 & 6 & b_{\l} & a_{\l}
\end{array}
\right ], \quad \forall K \in G, ~ \forall G \subset \mT_{h_{L - \l}}.
\label{eq:ElmMatAl_3D}
\end{align}
Thus, the bound for $\gamma_G$ at level $L-\l$ reads
\begin{align}
\bigl(\gamma^{(L-\l)}_{G}\bigr)^{2} = \dfrac{72(a_{\l} + b_{\l})}{(a_{\l} + 12)(a_{\l} - 6)(a_{\l} - b_{\l})}.
\label{eq:def_GammaL_3D}
\end{align}
The result \eqref{eq:bnd_GammaL} then follows by taking $\gamma^{L - \l}_{G} = c_{\l, D}$, where $c_{\l, D}$ is defined in Lemma~\ref{lem:seq_c2_3D}.
\end{proof}
\begin{remark}
The curves in Figure~\ref{fig:seq_c} show the behavior of $\gamma^{2}_{G}$ (defined by \eqref{eq:def_GammaL} and \eqref{eq:def_GammaL_3D}). We observe that $\gamma^{2}_{G}$ approaches zero when the splitting is applied many times (increasing $\l$ from left to right), which means that the two subspaces $\mV_{1}$ and $\mV_{2}$ in (\ref{Splitting:DS}) become increasingly orthogonal to each other as the recursion proceeds. Therefore, on (very) coarse levels, the upper bound $\Theta$ for $\gamma^{2}_{G}$, and thus for $\gamma^{2}$, is quite pessimistic.
\end{remark}
\begin{remark}
Note that the lowest order Raviart-Thomas (respectively Raviart-Thomas-Nedelec) type elements on general quadrilateral (respectively hexahedral) meshes do not show any convergence for the divergence of the field \cite{ArnoldBF-05}. In such cases, one can use, e.g., Arnold-Boffi-Falk type elements \cite{ArnoldBF-05}. However, the presented analysis won't suffice for such elements, and further work will be needed.
\end{remark}

\section{Algorithmic aspects}
\label{sec:Alg}
In this section we present the algorithms which have been used in this article for the solution of $M z = r$, the step used in preconditioned conjugate gradient method (PCG) for linear AMLI or flexible conjugate gradient method (FCG) for nonlinear AMLI. The algorithms, presented as pseudocodes with a compact syntax/style close to the \verb!matlab!$^{\circledR}$ language \cite{matlab}, should be helpful to the practitioners in the respective fields \footnote{The variable names listed in \textbf{Require} may be defined globally or passed as arguments.} . The preconditioner $M$, as explained in Section~\ref{sec:Prec}, requires the solution of nested systems $\hat{A} z = r$, and $B v = w$, where the matrices $\hat{A}$ and $B$ are defined in \eqref{eq:hat_A_LU} and \eqref{eq:B2x2}, respectively. Using the factorization \eqref{eq:hat_A_LU} we rewrite $\hat{A} z = r$ as follows
\begin{align}
\label{eq:TriangSys_A}
\left[ \begin{array}{cc}
\hA_{11} & 0  \\
\hA_{21} & B
\end{array} \right]
\left[ \begin{array}{c}
y_{1}  \\
y_{2}
\end{array} \right]
= 
\left[ \begin{array}{c}
r_{1}  \\
r_{2}
\end{array} \right], \quad
\left[ \begin{array}{cc}
I_1 & \hA_{11}^{-1} \hA_{12}  \\
0 & I_2
\end{array} \right]
\left[ \begin{array}{c}
z_{1}  \\
z_{2}
\end{array} \right]
= 
\left[ \begin{array}{c}
y_{1}  \\
y_{2}
\end{array} \right].
\end{align}
Similarly, using the partitioning \eqref{eq:B2x2} we rewrite $B v = w$ as follows
\begin{align}
\label{eq:TriangSys_B}
\left[ \begin{array}{cc}
B_{11} & 0  \\
B_{21} & B_{22}
\end{array} \right]
\left[ \begin{array}{c}
t_{1}  \\
t_{2}
\end{array} \right]
= 
\left[ \begin{array}{c}
w_{1}  \\
w_{2}
\end{array} \right], \quad
\left[ \begin{array}{cc}
I_3 & B_{11}^{-1} B_{12}  \\
0 & I_4
\end{array} \right]
\left[ \begin{array}{c}
v_{1}  \\
v_{2}
\end{array} \right]
= 
\left[ \begin{array}{c}
t_{1}  \\
t_{2}
\end{array} \right].
\end{align}
Note that in \eqref{eq:TriangSys_B} the matrix $B_{22}$ is an approximation of the exact Schur complement $S = B_{22} - B_{21} B_{11}^{-1} B_{12}$. Given the \emph{exact} $LU$ factors $L_{11}^{\hat{A}}$ and $U_{11}^{\hat{A}}$ of $\hat{A}_{11}$, and incomplete $LU$ factors $L_{11}^{B}$ and $U_{11}^{B}$ of $B_{11}$,  the Algorithms~\ref{algo:Solve_L} and \ref{algo:Solve_U} solve the triangular systems in \eqref{eq:TriangSys_A}-\eqref{eq:TriangSys_B}. Note that, since $v_{2} = t_{2}$, the solution of
\begin{align*}
B_{22} v_{2} = w_{2} - B_{21} t_{1} =: w_{c}
\end{align*}
is performed at the next coarser level with the recursive application of AMLI algorithm.
\begin{algorithm}[ht!]
\caption{Solve lower triangular system}
\label{algo:Solve_L}
\begin{algorithmic}
\Require $L_{11}^{\hat{A}}, U_{11}^{\hat{A}}, \hat{A}_{12}, L_{11}^{B}, U_{11}^{B}, B_{12}$
\Function{[$y_{1}, t_{1}, w_{c}$] = SolveL}{$r$}
\State $y_{1} = U_{11}^{\hat{A}} \backslash (L_{11}^{\hat{A}} \backslash r_{1})~; \quad
w = r_{2} - (\hat{A}_{12})^{T} y_{1}~;$
\hfill \Comment{See \eqref{eq:TriangSys_A} for the dimensions of $r_{1}$ and $r_{2}$}
\State $t_{1} = U_{11}^{B} \backslash (L_{11}^{B} \backslash w_{1})~;$
\hfill \Comment{See \eqref{eq:TriangSys_B} for the dimensions of $w_{1}$ and $w_{2}$}
\If {preconditioner is additive}
\State $w_{c} = w_{2}~;$
\Else
\State $w_{c} = w_{2} - (B_{12})^{T} t_{1}~;$
\EndIf
\EndFunction
\end{algorithmic}
\end{algorithm}
\begin{algorithm}[ht!]
\caption{Solve upper triangular system}
\label{algo:Solve_U}
\begin{algorithmic}
\Require $L_{11}^{\hat{A}}, U_{11}^{\hat{A}}, \hat{A}_{12}, L_{11}^{B}, U_{11}^{B}, B_{12}$
\Function{$z$ = SolveU}{$v_{2}, t_{1}, y_{1}$}
\If {preconditioner is additive}
\State $v_{1} = t_{1}~;$
\Else
\State $v_{1} = t_{1} - U_{11}^{B} \backslash (L_{11}^{B} \backslash (B_{12} v_{2}))~;$
\EndIf
\State $z_{2} = [v_{1}~;~ v_{2}]~; \quad
z_{1} = y_{1} - U_{11}^{\hat{A}} \backslash (L_{11}^{\hat{A}} \backslash (\hat{A}_{12} z_{2}))~; \quad
z = [z_{1}~;~ z_{2}]~;$
\EndFunction
\end{algorithmic}
\end{algorithm}
We now first present the algorithm for the linear AMLI method. This algorithm is adapted from \cite{AxelssonV-90, KrausMargenov-09, Vassilevski-08}. The linear AMLI algorithm requires the computation of coefficients $q_{i}, i = 0 \ldots \nu - 1$, from properly shifted and scaled Chebyshev polynomials. The algorithm presented below is for fixed $V$- or $\nu$-cycle for all levels ($\nu$-cycle also has the $V$-cycle at the finest level), which is commonly used in practice. For varying $V$- or $\nu$-cycles at any given level (and thus having more involved algorithm), see e.g., \cite[Alg.~10.1]{KrausMargenov-09}. \footnote{The vector ${d}^{(k-1)}$ in the right hand side of \cite[(10.6)]{KrausMargenov-09} is erroneous, and should be replaced by ${w}^{(k-1)}$, see \cite[(3.6)]{AxelssonV-90}.}
\begin{algorithm}[!ht]
\caption{Linear AMLI}
\label{algo:AMLI_L}
\begin{algorithmic}
\Require $\nu, q, J, B_{22}$
\Function{$z$ = LAMLI}{$r, L, \l$}
\State $r = J r~; \quad
[y_{1}, t_{1}, w_{c}] = \textsc{SolveL}(r)~;$
\If {$\l = L$} \hfill \Comment{Finest level, only $V$-cycle}
\State $r_{c} = w_{c}~; \quad
v_{2} = \textsc{SolveV2}(r_{c}, L, \l)~;$
\Else \hfill \Comment{Coarser levels, $V$- or $\nu$-cycle}
\State $r_{c} = q_{\nu - 1} w_{c}~; \quad
v_{2} = \textsc{SolveV2}(r_{c}, L, \l)~;$
\For{$\sigma = 2 : \nu$}
\State $r_{c} = B_{22} v_{2} + q_{\nu - \sigma} w_{c}~; \quad
v_{2} = \textsc{SolveV2}(r_{c}, L, \l)~;$
\EndFor
\EndIf
\State $z = \textsc{SolveU}(v_{2}, t_{1}, y_{1})~; \quad
z = J^{T} z~;$
\EndFunction
\Function{$v_{2}$ = SolveV2}{$r_{c}, L, \l$}
\If {$\l - 1 = 0$}
\State $v_{2} = B_{22} \backslash r_{c}~;$ \hfill \Comment{Exact solve at coarsest level}
\Else
\State $v_{2} = \mathrm{LAMLI(r_{c}, L, \l - 1)}~;$ \hfill \Comment{Recursive call to LAMLI for intermediate levels}
\EndIf
\EndFunction
\end{algorithmic}
\end{algorithm}
Finally, we present the nonlinear AMLI algorithm. This algorithm is adapted from \cite{AxelssonV-90, AxelssonV-94, KrausMargenov-09, Notay-00, Notay-02, Vassilevski-08}. Again, the algorithm presented below is for fixed $V$- or $\nu$-cycle for all levels, and thus has simpler presentation than for varying $V$- or $\nu$-cycles at any given level (see e.g., \cite[Alg.~10.2]{KrausMargenov-09}, \cite[Alg.~5.4]{AxelssonV-94} or \cite[Alg.~6.1]{Notay-02} for the latter). \footnote{The algorithm presented in \cite[Alg.~10.2]{KrausMargenov-09} recursively updates the vector $q$ in the \textbf{for} loop on $j$, which is not what was originally proposed in other two references.}
\begin{algorithm}[!ht]
\caption{Nonlinear AMLI}
\label{algo:AMLI_N}
\begin{algorithmic}
\Require $\nu, J, \hat{A}, B_{22}$
\Function{$z$ = NAMLI}{$r, L, \l$}
\State $z = 0~; \quad
r = J r~;$
\State $[y_{1}, t_{1}, r_{c}] = \textsc{SolveL}(r)~; \quad
v_{2} = \textsc{SolveV2}(r_{c}, L, \l)~;$
\If {$\l = L$} \hfill \Comment{Finest level, only $V$-cycle}
\State $p = \textsc{SolveU}(v_{2}, t_{1}, y_{1})~; \quad
z = z + p~;$
\Else \hfill \Comment{Coarser levels, $V$- or $\nu$-cycle}
\State $p_{1} = \textsc{SolveU}(v_{2}, t_{1}, y_{1})~; \quad
q_{1} = \hat{A} p_{1}~;$
\State $\tau_{1} = p_{1}^{T} q_{1}~; \quad
\alpha = (r^{T}, p_{1})/\tau_{1}~;$
\State $z = z + \alpha p_{1}~; \quad
r = r - \alpha q_{1}~;$
\For{$\sigma = 2 : \nu$}
\State $[y_{1}, t_{1}, r_{c}] = \textsc{SolveL}(r)~; \quad
v_{2} = \textsc{SolveV2}(r_{c}, L, \l)~; \quad
p_{\sigma} = \textsc{SolveU}(v_{2}, t_{1}, y_{1})~;$
\State $s = 0~;$
\For{$j = 1 : \sigma - 1$}
\State $\beta = (p_{\sigma}^{T} q_{j})/\tau_{j}~; \quad
s = s - \beta p_{j}~;$
\EndFor
\State $p_{\sigma} = p_{\sigma} + s~; \quad
q_{\sigma} = \hat{A} p_{\sigma}~;$
\State $\tau_{\sigma} = p_{\sigma}^{T} q_{\sigma}~; \quad
\alpha = (r^{T} p_{\sigma})/\tau_{\sigma}~;$
\State $z = z + \alpha p_{\sigma}~; \quad
r = r - \alpha q_{\sigma}~;$
\EndFor
\EndIf
\State $z = J^{T} z~;$
\EndFunction
\Function{$v_{2}$ = SolveV2}{$r_{c}, L, \l$}
\If {$\l - 1 = 0$}
\State $v_{2} = B_{22} \backslash r_{c}~;$ \hfill \Comment{Exact solve at coarsest level}
\Else
\State $v_{2} = \mathrm{NAMLI(r_{c}, L, \l - 1)}~;$ \hfill \Comment{Recursive call to NAMLI for intermediate levels}
\EndIf
\EndFunction
\end{algorithmic}
\end{algorithm}
\section{Numerical results}
\label{sec:NumRes}
All the numerical experiments presented in this section are performed using \verb!matlab!$^{\circledR}$ R2012b on an HP Z420 workstation with 12 core 3.2 GHz CPU and 64 GB RAM.
The initial guess is chosen as a zero vector, and the stopping criteria is chosen as $\epsilon \le 10^{-8}$, where $\epsilon$ and the average residual reduction factor $\rho$ are defined as
\[
\epsilon := \Vert r^{(n_{\rm it})} \Vert / \Vert r^{(0)} \Vert ~, \quad
\rho := \epsilon^{\frac{1}{n_{\rm it}}},
\]
and $n_{\rm it}$ is the number of iterations reported in the tables.
\subsection{Two-dimensional $\Hcurl$ problem}
\label{sec:NumRes_2D}
We first present numerical results for two-dimensional $\Hcurl$ problem. For all the numerical experiments, we consider a mesh of square elements of size $h=1/8, 1/64, \ldots, 1/2048$ (i.e., up to $8,392,704$ DOF for the finest level). We use a direct solver on the coarsest mesh that consists of $4 \times 4$ elements. Hence, the multilevel procedure is based on $1$ to $9$ levels of regular mesh refinement (resulting in an $\l$-level method, $\l = 3, \ldots, 11$).
\begin{example}
Consider the model problem \eqref{eq:ModProb} in a unit square, and fix the coefficients $\alpha = \beta = 1$. The problem data is chosen such that the exact solution is given by $\bu = (\pi \sin \pi x \cos \pi y, -\pi \cos \pi x \sin \pi y)^{T}$.
\end{example}
For the $W$-cycle method, we chose two-types of stabilization polynomials $q^{(\l)}$. One is based on Chebyshev polynomials (see, e.g., \cite{AxelssonV-90, KrausMargenov-09, Vassilevski-08}, denoted in the tables by $T$), for which the polynomial $q^{(\l)}(x)$ is defined as ${2}/{(s - b)} - {x}/{(s - b)^{2}}$, where $s = \sqrt{1 + b + b^{2} -\gamma^{2}}$, and $b$ is some constant estimating the upper bound of the condition number of preconditioned $B_{11}$ block, see the Appendix for details.
The other one is based on the polynomial of best uniform approximation to $1/x$ (see, e.g., \cite{KrausVZ-12}, denoted in the tables by $X$), for which the polynomial $q^{(\l)}(x)$ is defined as ${(2 - \gamma^{2})}/{(1 - \gamma^{2})} - {x}/{(1 - \gamma^{2})}$.
The results for the $V$-cycle and $W$-cycle multiplicative AMLI method are presented in Table~\ref{cap:ConvM_a1}. The second column confirms the error convergence behavior. We see that for decreasing $h$ the growth in the iteration number for $V$-cycle is moderate (as expected), whereas both the $W$-cycle versions ($T$ and $X$) exhibit $h$-independence. Moreover, the total time (factorization and solver) reported in eighth and eleventh columns also confirms that both the versions of $W$-cycle are of practical optimal complexity (slight increase in time may be attributed to the implementation issues). We note that in the multiplicative preconditioning the $X$-version $W$-cycle gives slightly better results than the $T$-version $W$-cycle.
\begin{table}[!htp]
\caption{Convergence results for multiplicative AMLI, $\alpha = \beta = 1$, $\chi = \bu - \buh$}
\label{cap:ConvM_a1}
\begin{center}
\begin{tabular}{|l|c|c|c|r|c|c|r|c|c|r|}
\hline
& & \multicolumn{3}{c|}{$V$-cycle} & \multicolumn{3}{c|}{$W$-cycle ($T$)} & \multicolumn{3}{c|}{$W$-cycle ($X$)}
\\
\hline
$1/h$ & $\Vert \curl \chi \Vert_{L^{2}(\Omega)}$ &
$n_{\rm it}$ & $\rho$ & $t_{\mathrm{sec}}$ &
$n_{\rm it}$ & $\rho$ & $t_{\mathrm{sec}}$ &
$n_{\rm it}$ & $\rho$ & $t_{\mathrm{sec}}$ \\
\hline
8 & 0.15946423 & 7 & 0.049 & 0.00 & 7 & 0.049 & 0.00 & 7 & 0.049 & 0.00 \\
16 & 0.08005229 & 8 & 0.094 & 0.01 & 8 & 0.083 & 0.01 & 8 & 0.094 & 0.01 \\
32 & 0.04006629 & 10 & 0.143 & 0.01 & 9 & 0.104 & 0.01 & 8 & 0.097 & 0.01 \\
64 & 0.02003817 & 11 & 0.174 & 0.04 & 9 & 0.105 & 0.05 & 8 & 0.100 & 0.04 \\
128 & 0.01001971 & 12 & 0.201 & 0.14 & 9 & 0.108 & 0.16 & 8 & 0.095 & 0.14 \\
256 & 0.00500993 & 13 & 0.224 & 0.54 & 9 & 0.109 & 0.55 & 8 & 0.088 & 0.51 \\
512 & 0.00250498 & 14 & 0.246 & 2.41 & 9 & 0.110 & 2.22 & 8 & 0.083 & 2.09 \\
1024 & 0.00125249 & 14 & 0.267 & 10.74 & 9 & 0.110 & 9.35 & 8 & 0.078 & 8.99 \\
2048 & 0.00062624 & 16 & 0.313 & 49.79 & 9 & 0.110 & 40.29 & 8 & 0.073 & 38.74 \\
\hline
\end{tabular}
\end{center}
\end{table}
We now test the AMLI method with additive preconditioning. The results for the $V$-cycle and both the $W$-cycle additive AMLI methods are presented in Table~\ref{cap:ConvA_a1}.
We also present the results for nonlinear variant of AMLI method, see e.g., \cite{AxelssonV-91, AxelssonV-94, Kraus-02, KrausMargenov-09, Notay-00, Notay-02, NotayV-08}, in the last three columns (denoted in the tables by $N$, $W$-cycle referring to two inner iterations). Surprisingly, in the additive form, the $T$-version $W$-cycle gives much better results than the $X$-version $W$-cycle, where the latter appears to be stabilizing only towards very fine mesh (many recursive levels). This can be attributed to the fact that for the additive preconditioning, for the choice of $\gamma = \sqrt{3/8}$, we require that $\nu > \sqrt{{(1 + \gamma)}/{(1 - \gamma)}} > 2$, which does not hold for (both) the $W$-cycle.
The results of nonlinear $W$-cycle further improve the results of $T$-version $W$-cycle (linear).
Since the nonlinear $W$-cycle AMLI method gives the best results (and is free from parameters $b$ and $\gamma$), in the remaining numerical experiments we will only present the results from multiplicative form of $V$-cycle and nonlinear $W$-cycle AMLI method.
\begin{table}[!htp]
\caption{Convergence results for additive AMLI, $\alpha = \beta = 1$}
\label{cap:ConvA_a1}
\begin{center}
\begin{tabular}{|l|c|c|r|c|c|r|c|c|r|c|c|r|}
\hline
& \multicolumn{3}{c|}{$V$-cycle} & \multicolumn{3}{c|}{$W$-cycle ($T$)}
& \multicolumn{3}{c|}{$W$-cycle ($X$)} & \multicolumn{3}{c|}{$W$-cycle ($N$)}
\\
\hline
$1/h$ &
$n_{\rm it}$ & $\rho$ & $t_{\mathrm{sec}}$ &
$n_{\rm it}$ & $\rho$ & $t_{\mathrm{sec}}$ &
$n_{\rm it}$ & $\rho$ & $t_{\mathrm{sec}}$ &
$n_{\rm it}$ & $\rho$ & $t_{\mathrm{sec}}$ \\
\hline
8 & 10 & 0.153 & 0.00 & 10 & 0.153 & 0.00 & 10 & 0.153 & 0.00 & 10 & 0.153 & 0.00 \\
16 & 17 & 0.300 & 0.01 & 17 & 0.299 & 0.01 & 17 & 0.299 & 0.01 & 12 & 0.208 & 0.01 \\
32 & 20 & 0.391 & 0.02 & 19 & 0.346 & 0.03 & 23 & 0.446 & 0.03 & 12 & 0.209 & 0.03 \\
64 & 25 & 0.472 & 0.06 & 19 & 0.372 & 0.08 & 31 & 0.550 & 0.13 & 12 & 0.197 & 0.08 \\
128 & 30 & 0.538 & 0.21 & 21 & 0.386 & 0.26 & 44 & 0.653 & 0.47 & 11 & 0.179 & 0.23 \\
256 & 34 & 0.575 & 0.87 & 19 & 0.377 & 0.79 & 56 & 0.712 & 1.82 & 11 & 0.167 & 0.76 \\
512 & 39 & 0.617 & 3.99 & 19 & 0.361 & 3.04 & 60 & 0.735 & 6.85 & 9 & 0.127 & 2.61 \\
1024 & 44 & 0.657 & 19.20 & 19 & 0.362 & 12.46 & 65 & 0.751 & 28.90 & 9 & 0.117 & 10.65 \\
2048 & 50 & 0.685 & 91.72 & 19 & 0.371 & 52.99 & 65 & 0.752 & 121.49 & 8 & 0.098 & 43.03 \\
\hline
\end{tabular}
\end{center}
\end{table}
\begin{example}
Consider the model problem \eqref{eq:ModProb} in a unit square, fix the coefficient $\beta = 1$ and take $\alpha = 10^{m_{0}}$ for $m_{0} = \{-6, -3, 0, 3, 6 \}$. The right hand side (RHS) vector is all ones.
\end{example}
The results for the multiplicative AMLI method for varying $\alpha$ are presented in Table~\ref{cap:ConvM_av} for $V$- and nonlinear $W$-cycle. We see that the $V$-cycle shows some effect of $\alpha$, with a moderate growth in the number of iterations for decreasing $h$, however, the nonlinear $W$-cycle is independent of $h$, and is fully robust with respect to $\alpha$. Note that towards very large values of $\alpha$, the system matrix is well-conditioned, and the hierarchical splitting approaches orthogonal decomposition, therefore, the $V$-cycle method also exhibits optimal order complexity.
\begin{table}[!htp]
\caption{Convergence results for multiplicative AMLI, $\beta = 1, \alpha = 10^{m_{0}}$}
\label{cap:ConvM_av}
\begin{center}
\begin{tabular}{|l|cc|cc|cc|cc|cc|}
\hline
& \multicolumn{10}{c|}{$n_{\rm it}$} \\
\hline
{$\alpha \rightarrow$} &
\multicolumn{2}{c|}{$10^{-6}$} & \multicolumn{2}{c|}{$10^{-3}$} &
\multicolumn{2}{c|}{$10^{0}$} & \multicolumn{2}{c|}{$10^{3}$} &
\multicolumn{2}{c|}{$10^{6}$} \\
\hline
{$1/h$} & $V$ & $W$ & $V$ & $W$ & $V$ & $W$ & $V$ & $W$ & $V$ & $W$ \\
\hline
8 & 9 & 9 & 9 & 9 & 9 & 9 & 4 & 4 & 2 & 2 \\
16 & 12 & 10 & 12 & 10 & 12 & 10 & 7 & 6 & 2 & 2 \\
32 & 15 & 10 & 15 & 10 & 14 & 10 & 9 & 8 & 2 & 2 \\
64 & 17 & 10 & 17 & 10 & 16 & 10 & 11 & 9 & 2 & 2 \\
128 & 20 & 9 & 20 & 9 & 17 & 9 & 12 & 9 & 3 & 3 \\
256 & 22 & 9 & 22 & 9 & 18 & 9 & 14 & 9 & 4 & 4 \\
512 & 26 & 9 & 26 & 9 & 21 & 9 & 16 & 9 & 6 & 6 \\
1024 & 28 & 9 & 28 & 9 & 23 & 9 & 17 & 9 & 8 & 8 \\
2048 & 28 & 9 & 31 & 8 & 25 & 8 & 20 & 8 & 10 & 8 \\
\hline
\end{tabular}
\end{center}
\end{table}
\begin{example}
Consider the model problem \eqref{eq:ModProb} in a unit square, fix the coefficient $\alpha = 1$ and take $\beta = 10^{m_{0}}$ for $m_{0} = \{-6, -3, 0, 3, 6 \}$. The RHS vector is all ones.
\end{example}
The results for the multiplicative AMLI method for varying $\beta$ are presented in Table~\ref{cap:ConvM_bv} for $V$- and $W$-cycles. The results are qualitatively the same as in Table~\ref{cap:ConvM_av} for varying $\alpha$, with the parameter value reversing the behavior of the solver.
\begin{table}
\caption{Convergence results for multiplicative AMLI, $\alpha = 1$}
\label{cap:ConvM_bv}
\begin{center}
\begin{tabular}{|l|cc|cc|cc|cc|cc|}
\hline
& \multicolumn{10}{c|}{$n_{\rm it}$} \\
\hline
{$\beta \rightarrow$} &
\multicolumn{2}{c|}{$10^{-6}$} & \multicolumn{2}{c|}{$10^{-3}$} &
\multicolumn{2}{c|}{$10^{0}$} & \multicolumn{2}{c|}{$10^{3}$} &
\multicolumn{2}{c|}{$10^{6}$} \\
\hline
{$1/h$} & $V$ & $W$ & $V$ & $W$ & $V$ & $W$ & $V$ & $W$ & $V$ & $W$ \\
\hline
8 & 2 & 2 & 4 & 4 & 9 & 9 & 9 & 9 & 9 & 9 \\
16 & 2 & 2 & 7 & 6 & 12 & 10 & 12 & 10 & 12 & 10 \\
32 & 2 & 2 & 9 & 8 & 14 & 10 & 15 & 10 & 15 & 10 \\
64 & 2 & 2 & 11 & 9 & 16 & 10 & 17 & 10 & 17 & 10 \\
128 & 3 & 3 & 12 & 9 & 17 & 9 & 20 & 9 & 20 & 9 \\
256 & 4 & 4 & 14 & 9 & 18 & 9 & 22 & 9 & 22 & 9 \\
512 & 6 & 6 & 16 & 9 & 21 & 9 & 26 & 9 & 26 & 9 \\
1024 & 8 & 8 & 17 & 9 & 23 & 9 & 28 & 9 & 28 & 9 \\
2048 & 10 & 8 & 20 & 8 & 25 & 8 & 31 & 8 & 28 & 9 \\
\hline
\end{tabular}
\end{center}
\end{table}
\begin{example}
Consider the model problem \eqref{eq:ModProb} in a unit square, and fix the coefficient $\beta = 1$. The coefficient $\alpha$ is chosen as $1$ in $[0,0.5]^{2} \bigcup (0.5,1]^{2}$ and $\kappa$ elsewhere, where $\kappa = 10^{m_{0}}$, and $m_{0} = \{-6, -4, -2, 0 \}$. The RHS vector is all ones.
\end{example}
Finally, the results for the multiplicative AMLI method for the case with jump in the coefficients (aligned with the coarsest level mesh), which are presented in Table~\ref{cap:ConvM_jump} for $V$- and nonlinear $W$-cycles, show robustness with respect to the jump in the coefficients.
\begin{table}
\caption{Convergence results for multiplicative AMLI with jump in the coefficients, $\beta = 1$}
\label{cap:ConvM_jump}
\begin{center}
\begin{tabular}{|l|cc|cc|cc|cc|}
\hline
& \multicolumn{8}{c|}{$n_{\rm it}$} \\
\hline
{$\kappa \rightarrow$} &
\multicolumn{2}{c|}{$10^{0}$} & \multicolumn{2}{c|}{$10^{-2}$} &
\multicolumn{2}{c|}{$10^{-4}$} & \multicolumn{2}{c|}{$10^{-6}$} \\
\hline
{$1/h$} & $V$ & $W$ & $V$ & $W$ & $V$ & $W$ & $V$ & $W$ \\
\hline
8 & 9 & 9 & 10 & 10 & 10 & 10 & 10 & 10 \\
16 & 12 & 10 & 12 & 11 & 13 & 11 & 13 & 11 \\
32 & 14 & 10 & 15 & 11 & 15 & 11 & 16 & 11 \\
64 & 16 & 10 & 17 & 11 & 18 & 11 & 19 & 11 \\
128 & 17 & 9 & 20 & 11 & 20 & 11 & 21 & 11 \\
256 & 18 & 9 & 22 & 10 & 22 & 11 & 24 & 11 \\
512 & 21 & 9 & 23 & 10 & 26 & 11 & 26 & 11 \\
1024 & 23 & 9 & 26 & 10 & 28 & 11 & 28 & 11 \\
2048 & 25 & 8 & 28 & 10 & 32 & 11 & 32 & 11 \\
\hline
\end{tabular}
\end{center}
\end{table}
\subsection{Three-dimensional $\Hdiv$ problem}
\label{sec:NumRes_3D}
We now present the numerical results for three-dimensional $\Hdiv$ problem. For all the numerical experiments, we consider a uniformly refined mesh of cubic elements of size $h=1/4, \ldots, 1/128$ (i.e., up to $6,340,608$ DOF for the finest level). We use a direct solver on the coarsest mesh that consists of $2 \times 2$ elements. Hence, the multilevel procedure is based on $1$ to $6$ levels of regular mesh refinement (resulting in an $\l$-level method, $\l = 2, \ldots, 7$).
\begin{example}
Consider the model problem \eqref{eq:ModProb} in a unit cube, and fix the coefficients $\alpha = \beta = 1$. The problem data is chosen such that the exact solution is given by $\bu = \nabla (\sin \pi x \sin \pi y \sin \pi z)$.
\end{example}
For the linear AMLI $W$-cycle, here we only use the stabilization polynomial $q^{(\l)}(x)$ based on Chebyshev polynomials (and thus omit the notation $T$).
The results for the $V$-cycle and $W$-cycle multiplicative AMLI method are presented in Table~\ref{cap:ConvM_a1_3D}. The second column confirms the error convergence behavior. We see that for decreasing $h$ the growth in the iteration number for $V$-cycle is moderate (as expected), whereas both the $W$-cycle versions (linear and nonlinear) exhibit $h$-independence. Moreover, the total time (setup and solver) reported in eighth and eleventh columns also confirms that both the versions of $W$-cycle are of practical optimal complexity (slight increase in time may be attributed to the implementation issues). We note that the nonlinear $W$-cycle gives better results than the linear $W$-cycle. As a comparison, in the last column we report the timings required for the direct solver in \verb!matlab!$^{\circledR}$, which exhibit $\mathcal{O}(N_{L}^{2})$ complexity against the optimal $\mathcal{O}(N_{L})$ complexity of the presented AMLI method.
\begin{table}[!ht]
\caption{Convergence results for multiplicative AMLI, $\alpha = \beta = 1$, $\chi = \bu - \buh$}
\label{cap:ConvM_a1_3D}
\begin{center}
\begin{tabular}{|l|c|c|c|r|c|c|r|c|c|r|r|}
\hline
& & \multicolumn{3}{c|}{$V$-cycle} & \multicolumn{3}{c|}{Linear $W$-cycle} & \multicolumn{3}{c|}{Nonlinear $W$-cycle}
& $A_{h} \backslash f_{h}$
\\
\hline
$1/h$ & $\Vert \dvg \chi \Vert_{L^{2}(\Omega)}$ &
$n_{\rm it}$ & $\rho$ & $t_{\mathrm{sec}}$ &
$n_{\rm it}$ & $\rho$ & $t_{\mathrm{sec}}$ &
$n_{\rm it}$ & $\rho$ & $t_{\mathrm{sec}}$ &
$t_{\mathrm{sec}}$ \\
\hline
4 & 0.37955365 & 8 & 0.0992 & $< 0.01$ & 8 & 0.0992 & $< 0.01$ & 8 & 0.0992 & $< 0.01$ & $< 0.01$ \\
8 & 0.19467752 & 10 & 0.1469 & 0.02 & 10 & 0.1464 & 0.02 & 9 & 0.1020 & 0.03 & 0.01 \\
16 & 0.09796486 & 12 & 0.2092 & 0.11 & 11 & 0.1869 & 0.10 & 9 & 0.1147 & 0.11 & 0.04 \\
32 & 0.04906112 & 14 & 0.2525 & 0.81 & 12 & 0.1995 & 0.79 & 8 & 0.0958 & 0.74 & 1.09 \\
64 & 0.02454041 & 15 & 0.2912 & 7.10 & 12 & 0.1925 & 6.65 & 7 & 0.0684 & 6.03 & 63.58 \\
128 & 0.01227144 & 17 & 0.3374 & 63.04 & 12 & 0.2004 & 56.12 & 7 & 0.0608 & 50.87 & 5082.70 \\
\hline
\end{tabular}
\end{center}
\end{table}
We now test the AMLI method with additive preconditioning. The results for the $V$-cycle and both the $W$-cycle additive AMLI methods are presented in Table~\ref{cap:ConvA_a1_3D}.
Note that for the additive preconditioning, for the choice of $\gamma = \sqrt{1/2}$, we require that $\nu > \sqrt{{(1 + \gamma)}/{(1 - \gamma)}} = 1 + \sqrt{2}> 2$. However, both the $W$-cycle methods (for $\nu = 2$) exhibit optimal order. This may be attributed to the special structure (and clustering of eigenvalues) of the problem.
The results of nonlinear $W$-cycle further improves the results of linear $W$-cycle (as compared to the multiplicative version).
Since the nonlinear $W$-cycle AMLI method gives the best results (and is free from parameters $b$ and $\gamma$), in the remaining numerical experiments we will only present the results from multiplicative form of $V$-cycle and nonlinear $W$-cycle AMLI method.
\begin{table}[!ht]
\caption{Convergence results for additive AMLI, $\alpha = \beta = 1$}
\label{cap:ConvA_a1_3D}
\begin{center}
\begin{tabular}{|l|c|c|r|c|c|r|c|c|r|}
\hline
& \multicolumn{3}{c|}{$V$-cycle} & \multicolumn{3}{c|}{Linear $W$-cycle} & \multicolumn{3}{c|}{Nonlinear $W$-cycle}
\\
\hline
$1/h$ &
$n_{\rm it}$ & $\rho$ & $t_{\mathrm{sec}}$ &
$n_{\rm it}$ & $\rho$ & $t_{\mathrm{sec}}$ &
$n_{\rm it}$ & $\rho$ & $t_{\mathrm{sec}}$ \\
\hline
4 & 12 & 0.2050 & $< 0.01$ & 12 & 0.2050 & $< 0.01$ & 12 & 0.2050 & $< 0.01$ \\
8 & 18 & 0.3592 & 0.02 & 20 & 0.3736 & 0.03 & 15 & 0.2883 & 0.03 \\
16 & 24 & 0.4640 & 0.12 & 28 & 0.5079 & 0.15 & 16 & 0.2951 & 0.13 \\
32 & 30 & 0.5380 & 1.03 & 27 & 0.5024 & 1.03 & 15 & 0.2840 & 0.86 \\
64 & 36 & 0.5948 & 9.53 & 28 & 0.5077 & 8.56 & 14 & 0.2578 & 6.91 \\
128 & 41 & 0.6329 & 85.38 & 28 & 0.5160 & 71.13 & 13 & 0.2347 & 56.61 \\
\hline
\end{tabular}
\end{center}
\end{table}
\begin{example}
Consider the model problem \eqref{eq:ModProb} in a unit cube, fix the coefficient $\beta = 1$ and take $\alpha = 10^{m_{0}}$ for $m_{0} = \{-6, -3, 0, 3, 6 \}$. The right hand side (RHS) vector is all ones.
\end{example}
The results for the multiplicative AMLI method for varying $\alpha$ are presented in Table~\ref{cap:ConvM_av_3D} for $V$- and nonlinear $W$-cycle. We see that the $V$-cycle shows some effect of $\alpha$, with a moderate growth in the number of iterations for decreasing $h$, however, the nonlinear $W$-cycle is independent of $h$, and is fully robust with respect to $\alpha$. Note that towards very large values of $\alpha$, the system matrix is well-conditioned, and the hierarchical splitting approaches orthogonal decomposition, therefore, the $V$-cycle method also exhibits optimal order complexity.
\begin{table}[!ht]
\caption{Convergence results for multiplicative AMLI, $\beta = 1, \alpha = 10^{m_{0}}$}
\label{cap:ConvM_av_3D}
\begin{center}
\begin{tabular}{|l|cc|cc|cc|cc|cc|}
\hline
& \multicolumn{10}{c|}{$n_{\rm it}$} \\
\hline
{$\alpha \rightarrow$} &
\multicolumn{2}{c|}{$10^{-6}$} & \multicolumn{2}{c|}{$10^{-3}$} &
\multicolumn{2}{c|}{$10^{0}$} & \multicolumn{2}{c|}{$10^{3}$} &
\multicolumn{2}{c|}{$10^{6}$} \\
\hline
{$1/h$} & $V$ & $W$ & $V$ & $W$ & $V$ & $W$ & $V$ & $W$ & $V$ & $W$ \\
\hline
4 & 12 & 12 & 12 & 12 & 11 & 11 & 3 & 3 & 1 & 1 \\
8 & 15 & 13 & 15 & 12 & 15 & 12 & 5 & 5 & 2 & 2 \\
16 & 18 & 13 & 18 & 13 & 18 & 13 & 8 & 8 & 2 & 2 \\
32 & 21 & 12 & 21 & 12 & 20 & 12 & 11 & 10 & 2 & 2 \\
64 & 23 & 12 & 24 & 12 & 24 & 12 & 14 & 11 & 2 & 2 \\
128 & 27 & 12 & 25 & 12 & 25 & 12 & 16 & 11 & 3 & 3 \\
\hline
\end{tabular}
\end{center}
\end{table}
Since fixing $\alpha$ and varying $\beta$ only reverses the behavior (from left to right) as presented in Table~\ref{cap:ConvM_av_3D}, see also Section~\ref{sec:NumRes_2D}, we do not include those results here.
\begin{example}
Consider the model problem \eqref{eq:ModProb} in a unit cube, and fix the coefficient $\beta = 1$. The coefficient $\alpha$ is chosen as $1$ in $[0,0.5]^{3} \bigcup (0.5,1]^{2} \times [0,0.5] \bigcup [0,0.5] \times (0.5,1]^{2} \bigcup (0.5,1] \times [0,0.5] \times (0.5,1]$ and $\kappa$ elsewhere, where $\kappa = 10^{m_{0}}$, and $m_{0} = \{-6, -4, -2, 0 \}$. The RHS vector is all ones.
\end{example}
Finally, the results for the multiplicative AMLI method for the case with jump in the coefficients (aligned with the coarsest level mesh), which are presented in Table~\ref{cap:ConvM_jump_3D} for $V$- and nonlinear $W$-cycles, also show robustness with respect to jumps in the coefficients.
\begin{table}[ht!]
\caption{Convergence results for multiplicative AMLI with jump in the coefficients, $\beta = 1$}
\label{cap:ConvM_jump_3D}
\begin{center}
\begin{tabular}{|l|cc|cc|cc|cc|}
\hline
& \multicolumn{8}{c|}{$n_{\rm it}$} \\
\hline
{$\kappa \rightarrow$} &
\multicolumn{2}{c|}{$10^{-6}$} & \multicolumn{2}{c|}{$10^{-4}$} &
\multicolumn{2}{c|}{$10^{-2}$} & \multicolumn{2}{c|}{$10^{0}$} \\
\hline
{$1/h$} & $V$ & $W$ & $V$ & $W$ & $V$ & $W$ & $V$ & $W$ \\
\hline
4 & 13 & 13 & 13 & 13 & 12 & 12 & 11 & 11 \\
8 & 18 & 15 & 17 & 14 & 16 & 13 & 15 & 12 \\
16 & 23 & 13 & 20 & 13 & 19 & 13 & 18 & 13 \\
32 & 26 & 13 & 24 & 13 & 22 & 13 & 20 & 12 \\
64 & 29 & 13 & 27 & 13 & 25 & 13 & 24 & 12 \\
128 & 33 & 13 & 30 & 13 & 28 & 13 & 25 & 12 \\
\hline
\end{tabular}
\end{center}
\end{table}
\section{Conclusion}
\label{sec:Conclusion}

We have presented an optimal order AMLI method for problems in two-dimensional $\Hcurl$ space and three-dimensional $\Hdiv$ space. In the hierarchical setting, we derived explicit recursion formulae to compute the element matrices, and bounds for the multilevel behavior of $\gamma$ that are robust with respect to the coefficients in the model problem. The main result of our local analysis (Theorem~\ref{thm:Gamma}) shows that a second order stabilization polynomial (or two inner iterations in nonlinear method), i.e., a $W$-cycle, is sufficient to stabilize the AMLI process.
The presented numerical results, including the case with jumping coefficients (aligned with the coarsest level mesh) confirm the robustness and efficiency of the proposed method.
The performance of the presented methods for the range of parameters considered in the paper shows that these methods can be effectively used by the practitioners in the respective fields.
\begin{acknowledgements}
The author is very grateful to Dr. Johannes Kraus (RICAM, Linz) for insightful discussions on AMLI methods. Thanks are also due to Dr. Christoph Koutschan (RICAM, Linz) for helpful discussions in proving Lemma~\ref{lem:seq_c2_3D}.
\end{acknowledgements}

\appendix
\section{Coefficients of polynomial $q$}
In this appendix, we briefly discuss the computation of the polynomial coefficients for linear AMLI $W$-cycle. In \cite[pp. 1582-83]{AxelssonV-90}, authors provided the explicit formulae for the computation of the coefficients of the polynomial $q_{\nu}$, for polynomial degrees $\nu = 2,3$. Note that $q_{\nu}$ is a polynomial of degree $\nu - 1$. Since only the $W$-cycle is used in this paper, we discuss only the $\nu = 2$ case, i.e. $q(x) = q_{0} + q_{1} x$. Given the constants $\gamma$ and $b$ (which measures the quality of approximation of $A_{11}$ by $C_{11}$), the Algorithm~\ref{algo:q_AV90} computes the coefficients $q_{0}$ and $q_{1}$.
\begin{algorithm}[!ht]
\caption{Coefficients of $q(x)$, see \cite[pp. 1582-1583]{AxelssonV-90}}
\label{algo:q_AV90}
\begin{algorithmic}
\State $T_{2}(x) = 2 x^{2} - 1$
\State $\alpha = (3 - 4 \gamma^{2})/\left ( 1 + 2b + \sqrt{3 - 4 \gamma^{2} + (1 + 2b)^{2}} \right )$
\State $a = (1 + \alpha)/(1 - \alpha)$, $c = 1/(1 + T_{2}(a))$
\State $q_{0} = 8 a c/(1 - \alpha), \quad q_{1} = -8c/(1 - \alpha)^{2}$~.
\end{algorithmic}
\end{algorithm}
To simplify the expressions, we introduce a variable $s = \sqrt{1 - \gamma^{2} + b + b^{2}}$, which gives $1 - \gamma^{2} = s^{2} - b - b^{2}$. Now
\begin{align*}
\alpha & = \dfrac{3 - 4 \gamma^{2}}{1 + 2b + \sqrt{3 - 4 \gamma^{2} + (1 + 2b)^{2}}}
= \dfrac{4(1 - \gamma^{2}) - 1}{1 + 2b + \sqrt{4(1 - \gamma^{2}) + 4(b + b^{2})}} \\
& = \dfrac{4(s^{2} - b - b^{2}) - 1}{1 + 2b + 2s}
= \dfrac{4s^{2} - (1 + 2b)^{2}}{1 + 2b + 2s}
= 2s - 2b -1.
\end{align*}
Therefore, $1 + \alpha = 2s - 2b$. From the relations of $a$ and $c$, we have $a (1 - \alpha) = 1 + \alpha = 2s - 2b$, and $c = 1/(2a^{2})$. Using these simplifications, we get
\begin{align*}
q_{0} & = \dfrac{8 a c}{1 - \alpha} = \dfrac{4}{a (1 - \alpha)} = \dfrac{2}{s - b},\\
q_{1} & = -\dfrac{8c}{(1 - \alpha)^{2}} = -\dfrac{4}{a^{2} (1 - \alpha)^{2}} = -\dfrac{1}{(s - b)^{2}}.
\end{align*}
Therefore, we can write the steps of Algorithm~\ref{algo:q_AV90} in simplified form as follows:
\begin{align}
\label{eq:q_simple}
s = \sqrt{1 - \gamma^{2} + b + b^{2}}, \quad q_{0} = \dfrac{2}{s - b}, \quad q_{1} = -\dfrac{1}{(s - b)^{2}}.
\end{align}
Note that, in several practical applications, see e.g., \cite{GKrausM-08, GKrausM-08b, KrausMargenov-09}, the choice of $b = 0$, which yields $q_{0} = 2/\sqrt{1 - \gamma^{2}}$ and $q_{1} = -1/(1 - \gamma^{2})$, has been used. However, it is observed from the results in this paper that small negative values for $b$ can outperform the results for $b = 0$.

\end{document}